\documentclass[onefignum,onetabnum]{siamart190516}

\usepackage{amsmath}
\usepackage{amssymb}
\usepackage{amstext}
\usepackage{amscd}
\usepackage{arydshln}
\usepackage{color}
\usepackage{graphicx}
\usepackage{longtable}
\usepackage{calc}
\usepackage{microtype}
\usepackage{mathdots}

\newcommand{\diag}[1]{\mbox{${\rm diag}(#1)$}}

\definecolor{myred}{cmyk}{0.000000,1.000000,1.000000,0.1}
\definecolor{myblue}{cmyk}{1.000000,0.750000,0.000000,0.1}

\newcommand{\rank}[1]{\mbox{${\rm rank}(#1)$}}
\newcommand{\BC}{\mathbb{C}}

\newcommand{\BR}{\mathbb{R}}
\newcommand{\la}{\lambda}

\newcommand{\adots}{.\:\!\raisebox{.6ex}{.}\:\!\raisebox{1.2ex}{.}}

\newcommand{\hide}[1]{}

\newsiamremark{remark}{Remark}
\newsiamremark{prop}{Proposition}
\newtheorem{example}[theorem]{Example}

\newcounter{parenumi}
  {\end{list}}

\headers{Strongly minimal self-conjugate linearizations}{F.M. Dopico, M.C. Quintana, and P. Van Dooren}

\title{Strongly minimal self-conjugate linearizations for polynomial and rational matrices\thanks{Submitted to the editors DATE. The results of this paper are included in Chapter 8 of the second author's PhD Thesis defended in September 28, 2021 at Universidad Carlos III de Madrid.
\funding{This publication is part of the ``Proyecto de I+D+i PID2019-106362GB-I00 financiado por MCIN/AEI/10.13039/501100011033''. It has been also funded by
``Ministerio de Econom\'ia, Industria y Competitividad (MINECO)'' of Spain  through grant MTM2017-90682-REDT and by the Madrid Government (Comunidad de Madrid-Spain) under the Multiannual Agreement with UC3M in the line of Excellence of University Professors (EPUC3M23), in the V PRICIT (Regional Programme of Research and Technological Innovation). Mar\'ia C. Quintana was funded by the ``contrato predoctoral'' BES-2016-076744 of MINECO
and by an Academy of Finland grant (Suomen Akatemian p\"{a}\"{a}t\"{o}s 331240). This work was partially developed while the third author held a ``Chair of Excellence UC3M - Banco de
Santander'' at Universidad Carlos III de Madrid in the academic year 2019-2020.}}}

\author{Froil\'an M. Dopico\thanks{Departamento de Matem\'aticas, Universidad Carlos III de Madrid, Avda. Universidad 30, 28911 Legan\'es, Spain (\email{dopico@math.uc3m.es}).}
\and Mar\'ia C. Quintana\thanks{Department of Mathematics and Systems Analysis, Aalto University, Otakaari 1, Espoo, Finland (\email{maria.quintanaponce@aalto.fi}).}
\and Paul~Van~Dooren\thanks{Department of Mathematical Engineering, Universit\'{e} catholique de Louvain, Avenue Georges Lema\^{i}tre 4, B-1348 Louvain-la-Neuve, Belgium (\email{paul.vandooren@uclouvain.be}).}
}

\begin{document}

\maketitle

\begin{abstract}
We prove that we can always construct strongly minimal linearizations of an arbitrary rational matrix from its Laurent expansion around the point at infinity, which happens to be the case for polynomial matrices expressed in the monomial basis.  If the rational matrix has a particular self-conjugate structure we show how to construct strongly minimal linearizations that preserve it. The structures that are considered are the Hermitian and skew-Hermitian rational matrices with respect to the real line,
and the para-Hermitian and para-skew-Hermitian matrices with respect to the imaginary axis. We pay special attention to the construction of strongly minimal linearizations for the particular case of structured polynomial matrices. The proposed constructions lead to efficient numerical algorithms for constructing strongly minimal linearizations. The fact that they are valid for {\em any} rational matrix is an improvement on any other previous approach for constructing other classes of structure preserving linearizations, which are not valid for any structured rational or polynomial matrix. The use of the recent concept of strongly minimal linearization is the key for getting such generality.

Strongly minimal linearizations are Rosenbrock's polynomial system matrices of the given rational matrix, but with a quadruple of  linear polynomial matrices (i.e. pencils)~:
	$$L(\la):=\left[\begin{array}{ccc} A(\lambda) & -B(\lambda) \\ C(\lambda) & D(\lambda) \end{array}\right],$$
	where $A(\lambda)$ is regular, and the pencils $
	\left[\begin{array}{ccc} A(\la) & -B(\la)  \end{array}\right]$ and $ \left[\begin{array}{ccc} A(\la) \\ C(\la)   \end{array}\right]$ have no finite or infinite eigenvalues. Strongly minimal linearizations contain the complete information about the zeros, poles and minimal indices of the rational matrix and allow to recover very easily its eigenvectors and minimal bases. Thus, they can be combined with algorithms for the generalized eigenvalue problem for computing the complete spectral information of the rational matrix.
\end{abstract}



\section{Introduction}
\label{sec:Introduction}

In the seventies, Rosenbrock \cite{Ros70} introduced the concept of a polynomial system matrix $L(\la)$ of an arbitrary rational matrix $R(\la)\in\BC(\la)^{m\times n}$.
Such a system matrix is partitioned in a quadruple $\{A(\la),B(\la),C(\la),D(\la)\}$ of compatible polynomial matrices
\begin{equation} \label{pencil} L(\la):=\left[\begin{array}{ccc} A(\la) & -B(\la) \\ C(\la) & D(\la)
\end{array}\right]
\end{equation}
such that its Schur complement with respect to $D(\la)$ equals $R(\la)$. That is, $R(\la)= D(\la) + C(\la)A(\la)^{-1}B(\la)$. Then the quadruple $\{A(\la),B(\la),C(\la),D(\la)\}$ is said to be a realization of $R(\la)$. Rosenbrock showed that one can retrieve from the polynomial matrices $A(\la)$ and $L(\la)$, respectively, the finite pole and zero structure of $R(\la)$, provided $L(\la)$ is {\em irreducible} or {\em minimal}, meaning that the matrices
\begin{equation}  \label{minimal}
\left[\begin{array}{ccc} A(\la) & {\color{blue}-}B(\la)  \end{array}\right], \quad \left[\begin{array}{ccc} A(\la) \\ C(\la)   \end{array}\right],
\end{equation}
have, respectively, full row and column rank for all finite $\la\in \BC$. It was shown recently in \cite{DQV} that when the quadruple consists of polynomial matrices of degree at most one, i.e. pencils, then one can recover the complete eigenstructure of the rational matrix, namely its finite and infinite polar and zero structure, and its left and right null space structure from the pencils $A(\la)$ and $L(\la)$ provided the pencils in \eqref{minimal} have full rank for all $\la$, infinity included. Moreover, in this situation, the eigenvectors and minimal bases of $R(\la)$ can be very easily recovered from those of $L(\la)$, and their minimal indices are the same. In such a case, $L(\la)$ is said to be {\em strongly minimal} \cite{DMQVD,DQV} or, also, a {\em strongly minimal linearization of $R(\la)$}.
The main advantage of using pencils is that there are well-established stable algorithms to
compute their eigenstructure using unitary transformations only, both in the regular \cite{moler-stewart} and in the singular \cite{Van79b} case. There are also algorithms available to derive strongly minimal linear polynomial system matrices, from non-minimal ones. These algorithms are also based on unitary transformations only \cite{Van81,DQV}.

In this paper, we show how to construct strongly minimal linearizations for rational matrices $R(\la)\in\BC(\la)^{m\times n}$ starting from a Laurent expansion around the point at infinity~:
\begin{equation} \label{Laurent}
R(\la)= R_d \la^d +\cdots + R_1\la + R_0 + R_{-1}\la^{-1} + R_{-2}\la^{-2} + R_{-3}\la^{-3} + \cdots,
\end{equation}
which is convergent for sufficiently large $\la\in \BC$. The approach we propose is also valid if,  instead of considering the Laurent expansion $R_{-1}\la^{-1} + R_{-2}\la^{-2} + R_{-3}\la^{-3} + \cdots$ for the strictly proper part of $R(\la)$, {\em any} minimal state-space realization of the strictly proper part of $R(\la)$ is given.

If the rational matrix is square (i.e. $m=n$) and has a particular type of self-conjugate structure the coefficients $R_{i}\in\BC^{m\times m}$ of its expansion also inherit the self-conjugate structure and the poles and zeros of $R(\la)$ appear in self-conjugate pairs. Such structures arise in many applications, as we comment below, and in these cases we also show how to construct strongly minimal linearizations preserving the structure. In particular, we consider here four types of self-conjugate rational matrices, two with respect to the real line and two with respect to the imaginary axis. The Hermitian and skew-Hermitian rational matrices $R(\la)$, with respect to the real line, satisfy
$$ [R(\la)]^* = R(\overline \la), \quad \mathrm{and} \quad  [R(\la)]^* = -R(\overline \la),
$$
respectively. They have poles and zeros that are mirror images with respect to the real line $\BR$, and have coefficient matrices $R_i$ that are
Hermitian (i.e. $R_{i}^*=R_{i}$) and skew-Hermitian (i.e. $R_{i}^*=-R_{i}$), respectively.
The para-Hermitian and para-skew-Hermitian rational matrices, with respect to the imaginary axis, satisfy
$$  [R(\la)]^* = R(-\overline \la), \quad \mathrm{and} \quad  [R(\la)]^* = -R(-\overline \la),
$$
respectively. They have poles and zeros that are mirror images with respect to the  imaginary line $\jmath\BR$, and have scaled coefficient matrices $\jmath^i R_i$ that are Hermitian and skew-Hermitian, respectively. The nomenclature introduced above is used in the linear systems and control theory literature (see, for instance, \cite{GHNSVX02,ran2,ran1} and the references therein). However, in standard references on structured polynomial matrices \cite{GoodVibrations,MMMM-alternating} the para-Hermitian and para-skew-Hermitian structures are called alternating structures, because the matrix coefficients satisfy, respectively, $R_i^* = (-1)^i R_i$ and $R_i^* = (-1)^{i+1} R_i$ and, thus, alternate between being Hermitian or skew-Hermitian matrices. Specifically, para-Hermitian polynomial matrices are called $*$-even and para-skew-Hermitian polynomial matrices are called $*$-odd in \cite{GoodVibrations,MMMM-alternating}.

 There are of course equivalent definitions for real rational matrices, where all coefficient matrices $R_i$ are real. Namely, (skew-)symmetric and para-(skew-)symmetric rational matrices. In these cases, the poles and zeros satisfy the same symmetries that have been described above.

The symmetries in the zeros and poles of structured polynomial and rational matrices reflect specific physical properties, as they originate usually from the physical symmetries of the underlying applications \cite{GHNSVX02,hmmt2006,Lancaster_paper,GoodVibrations,rail1}. Such special structures occur in numerous applications  in engineering, mechanics, control, and linear systems theory. Some of the most common algebraic structures that appear in applications are the (skew-)symmetric (or Hermitian), and  the para-(skew-)symmetric (or Hermitian) or alternating structures considered in this work (see \cite{hmmt2006,GoodVibrations,ran2,ran1} and the references therein). For instance,
symmetric (or Hermitian) matrix polynomials arise in the classical problem of vibration analysis \cite{Hermitian1,Lancaster66}, and alternating matrix polynomials find applications in the study of corner singularities in anisotropic elastic materials \cite{corner3} and in the study of gyroscopic systems \cite{Lancaster_paper}. Rational matrices with the structures mentioned above have appeared, for instance, in the continuous-time linear-quadratic optimal control problem and in the spectral factorization problem \cite{GHNSVX02,ran2,ran1,Van81b}.

Because of the numerous applications where structured rational and polynomial  matrices occur, there have been many attempts to construct
linearizations for such structured rational and polynomial matrices that display the same structure as that of the rational or polynomial matrix (see \cite{Greeks2,das-alam-affine,DasAlamArXiv,DMQ,PartI,GHNSVX02, hmmt2006, Lancaster66,GoodVibrations} among many other references on this topic). An important motivation for this search is to preserve numerically in floating point arithmetic the symmetries of the zeros and poles of these structured problems by applying structured algorithms for structured generalized eigenvalue problems to these structured linearizations \cite{skew-staircase,Implicit_palQR,Jacobi_Hermitian,Krylov_symmetric,corner3,Schroder_thesis}. However, all these earlier attempts to construct structured linearizations find obstacles when they are applied to the structured problems considered in this paper, because they either cover only a subclass of the structures, or they impose certain conditions on the rational and polynomial matrices for their construction to apply, such as regularity, strict properness or invertibility of certain matrix coefficients. We emphasize that, for some polynomial matrices, the mentioned obstacles cannot be overcome in any way with the previously adopted definitions of linearization, because it has been proved in \cite{GoodVibrations,MMMM-alternating} that there exist alternating polynomial matrices which cannot be linearized at all according to the standard definitions of linearizations in \cite{GoodVibrations,MMMM-alternating}. In contrast, in the present paper, we give a construction of structured {\em strongly minimal linearizations} valid  for arbitrary rational and polynomial matrices, with any of the above four structures. Moreover, the proof used for this construction is different from these earlier papers, and we claim it to be simpler as well.

The paper is organized as follows. In Section \ref{Background}, we develop background material for the problem and introduce strongly minimal linearizations for polynomial and rational matrices. In Section \ref{sec:linearization_polynomial}, we show how to construct strongly minimal linearizations of arbitrary polynomial matrices, paying particular attention to quadratic polynomial matrices in Subsection \ref{subsec.lowrank}. In Section \ref{sec:self-cong-poly}, we extend this construction to structured strongly minimal linearizations of structured polynomial matrices. In Sections \ref{sec:sproper} and \ref{sec:selft-conjugate-sproper}, we develop analogous results for strictly proper rational matrices. That is, we build strongly minimal linearizations for arbitrary and structured strictly proper rational matrices, respectively. In Section \ref{sec:rational}, we combine the results in previous sections to construct strongly minimal linearizations for arbitrary and structured rational matrices. Finally, in Section \ref{sec:alg}, we comment some algorithmic aspects and, in Section \ref{sec:conclusions}, we give some concluding remarks and some lines of possible future research.

\section{Background and strongly minimal linearizations} \label{Background}

This section recalls basic definitions that are used throughout the paper and discusses the recent concept of strongly minimal linearizations of rational matrices \cite{DMQVD,DQV}, which is fundamental in this work. We refer to \cite{Kai80,Ros70} for more details.

We consider the field of complex numbers $\mathbb{C}$. Then $\mathbb{C}[\la]^{m\times n}$ and $\mathbb{C}(\la)^{m\times n}$ denote the sets of $m\times n$ matrices whose entries are in the ring of polynomials $\mathbb{C}[\la]$ and in the field of rational functions $\mathbb{C}(\la)$, respectively. Their elements are called polynomial and rational matrices.

A rational function $r(\la)=\frac{n(\la)}{d(\la)}$ is said to be proper if $\deg(n(\la))\leq\deg(d(\la))$ and strictly proper if $\deg(n(\la))<\deg(d(\la))$, where $\deg (\cdot)$ stands for degree. A (strictly) proper rational matrix is a matrix whose entries are (strictly) proper rational functions.
	By the division algorithm for polynomials, any rational function $r(\la)\in\BC(\la)$ can be uniquely written as $r(\la)=p(\la)+r_{sp}(\la),$ where $p(\la)$ is a polynomial and $r_{sp}(\la)$ a strictly proper rational function. Therefore, any rational matrix $R(\la)\in\BC(\la)^{m\times n}$ can be uniquely written as
	\begin{equation}\label{eq.polspdec}
	R(\la)=P(\la)+R_{sp}(\la)
	\end{equation}
	where $P(\la)\in\BC[\la]^{m\times n}$ is a polynomial matrix and $R_{sp}(\la)\in \BC(\lambda)^{m\times n}$ is a strictly proper rational matrix. Then, $P(\la)$ is called the polynomial part of $R(\la)$ and $R_{sp}(\la)$ the strictly proper part of $R(\la)$.

A rational matrix $M(\la)\in \mathbb{C}(\la)^{m\times n}$ is regular if it is square and its determinant is not identically equal to $0$. Otherwise, $M(\la)$ is said to be singular. A square rational matrix $M(\la)\in \mathbb{C}(\la)^{m\times m}$ is regular at a point $\la_0\in \mathbb{C}$ if $M(\la_0)$ is invertible, with $M(\la_0)\in \mathbb{C}^{m\times m}.$ $M(\la)$ is regular at infinity or biproper if $M(1/\la)$ is regular at $0.$ If $M(\la)$ is regular for all $\la_0\in \mathbb{C}$ then $M(\la)$ is said to be unimodular and, equivalently, it is a polynomial matrix with constant nonzero determinant. The normal rank of a rational matrix is the size of its largest nonidentically zero minor.

Poles and zeros of rational matrices are defined via the local Smith--McMillan form \cite{vandooren-laurent-1979,AmMaZa15}. Let $R(\la)\in\mathbb{C}(\la)^{m\times n}$ be a rational matrix of normal rank $r$ and let $\lambda_0\in \mathbb{C}$. Then there exist rational matrices $M_{\ell}(\lambda)$ and $M_r(\lambda)$ regular at $\lambda_0$ such that
  \begin{equation} \label{snf}
    M_{\ell}(\lambda)R(\lambda)M_r(\lambda) = \diag{(\la-\la_{0})^{d_{1}} ,\ldots, (\la-\la_{0})^{d_{r}}, 0_{(m-r)\times (n-r)}}  ,
  \end{equation}
  where $d_1\le d_2 \le \cdots \le d_r$ are integer numbers. The diagonal matrix in \eqref{snf} is unique and is called the {\em local Smith-McMillan form} of $R(\lambda)$ at $\lambda_0$. The exponents $d_i$ are called the
  {\em structural indices} of $R(\lambda)$ at $\lambda_0$. If there are strictly positive indices in \eqref{snf} and they are $0 < d_p\le \cdots \le d_r$, then $\la_0$ is a zero of $R(\la)$ with partial multiplicities $(d_p, \ldots , d_r)$. In this case, we also say that $(d_p, \ldots , d_r)$ is the zero structure of $R(\la)$ at $\la_0$. If there are strictly negative indices  in \eqref{snf} and they are $d_1\le \cdots \le d_q <0$, then $\la_0$ is a pole of $R(\la)$ with partial multiplicities $(-d_1, \ldots , -d_q)$. In this case, we also say that $(-d_1, \ldots , -d_q)$ is the pole (or polar) structure of $R(\la)$ at $\la_0$. If $\lambda_0=\infty$, then the factor $(\lambda-\lambda_0)$ is replaced by $\frac{1}{\lambda}$ in \eqref{snf}, the matrices $M_{\ell}(\lambda)$ and $M_r(\lambda)$ are biproper, and the structural indices, zeros, poles, as well as their partial multiplicities, of $R(\la)$ at infinity are defined analogously. Observe that the structural indices and the pole and the zero structures of $R(\la)$ at infinity are exactly those of $R(1/\la)$ at zero.

  The zero structure of a rational matrix $R(\la)\in\mathbb{C}(\la)^{m\times n}$ is comprised by the set of its zeros (finite and infinite) and their partial multiplicities. The sum of the partial multiplicities of all the zeros (finite and infinite) of $R(\la)$ is called the zero degree $\delta_z (R)$ of $R(\la)$. The pole (or polar) structure of a rational matrix $R(\la)\in\mathbb{C}(\la)^{m\times n}$ is comprised by the set of its poles (finite and infinite) and their partial multiplicities. The sum of the partial multiplicities of all the poles (finite and infinite) of $R(\la)$ is called the polar degree $\delta_p (R)$ of $R(\la)$, or, also, the McMillan degree of $R(\la)$  \cite{Kai80}.

  \begin{remark}\rm Polynomial matrices are particular cases of rational matrices. Therefore, the definitions above can be applied to polynomial matrices. However, standard literature on polynomial matrices \cite{Gan59,GohbergLancasterRodman09} use the term {\em eigenvalues} instead of {\em zeros and poles} and define the {\em structure at infinity} in a different way. We discuss these points in this remark. Note first that a polynomial matrix $P(\la)$ does not have finite poles, i.e., all the indices $d_i$ in \eqref{snf} are nonnegative for any finite $\la_0$. The finite eigenvalues of $P(\la)$ and their partial multiplicities \cite{GohbergLancasterRodman09} are exactly the same as the finite zeros of $P(\la)$ and their partial multiplicities. However, in \cite{GohbergLancasterRodman09}, a polynomial matrix $P(\la)$ of degree $d$ and normal rank $r$ is said to have an eigenvalue at infinity with partial multiplicities $0 < t_p \leq \cdots \leq t_r$ if the {\em reversal} polynomial matrix $\text{rev}_d P(\la) := \la^d P(1/\la)$ has an eigenvalue at $0$ with partial multiplicities $0 < t_p \leq \cdots \leq t_r$. In this situation the structural indices \eqref{snf} of $P(\la)$ at infinity when viewed as a rational matrix are
  \begin{equation} \label{eq.infstruct22}
	(d_{1}, d_{2},\ldots , d_{r})=(\underbrace{0,\ldots,0}_{p-1}, t_p ,\ldots, t_r) - (d,d,\ldots, d).
	\end{equation}
  Thus, the pole-zero structures of a polynomial matrix at infinity are different from its ``eigenvalue structure'' at infinity defined through the reversal, but they are easily related through \eqref{eq.infstruct22} and are completely equivalent to each other. From now on, we will make a clear distinction for any polynomial matrix $P(\la)$ of degree $d$: whenever we talk about its ``{\em eigenvalue} structure at infinity'', we refer to the zero structure of $\text{rev}_d P(\la)$ at $0$, and whenever we talk about its ``{\em pole or zero} structures at infinity'', we refer to the pole or zero structures of $P(1/\la)$ at $0$. Recall that such a distinction is not necessary at finite points. Moreover, we emphasize that a polynomial matrix of degree $d>0$ may or may not have an eigenvalue at infinity, may or may not have a zero at infinity, but {\em always has a pole at infinity} with largest partial multiplicity (or order) $d$. More on this topic can be found in \cite{ADHM19}. Finally, note that for pencils, i.e., polynomial matrices with degree $1$, the definition of ``eigenvalue structure at infinity'' via reversals is equivalent to that coming from the Kronecker canonical form \cite{Gan59}, and that the relation \eqref{eq.infstruct22} was pointed out in \cite{Van81}.
  	\end{remark}

  In addition to the pole and zero structures, a singular rational matrix has a singular structure or minimal indices. In order to define them, recall that every rational vector subspace $\mathcal{V}$,
  i.e., every subspace $\mathcal{V} \subseteq \mathbb{C}(\la)^n$ over the field $\mathbb{C}(\la)$,
  has bases consisting entirely of polynomial vectors. We call them polynomial bases.
  By Forney \cite{For75}, a minimal basis of $\mathcal{V}$ is a polynomial basis of $\mathcal{V}$ consisting of polynomial vectors whose sum of degrees is minimal among all polynomial bases of $\mathcal{V}$.
  Though minimal bases are not unique, the ordered list of degrees of the polynomial vectors in any minimal basis of $\mathcal{V}$ is unique. These degrees are called the minimal indices of $\mathcal{V}$.

We now consider a rational matrix $R(\la)\in\mathbb{C}(\la)^{m\times n}$ and the rational vector subspaces:
\[
\begin{array}{l}
\mathcal{N}_r (R)=\{x(\la)\in\mathbb{C}(\la)^{n\times 1}: R(\la)x(\la)=0\}, \text{ and}\\
\mathcal{N}_\ell (R)=\{y(\la)^{T}\in\mathbb{C}(\la)^{1\times m}: y(\la)^T R(\la)=0\},
\end{array}
\]
which are called the right and left null-spaces of $R(\la)$, respectively. If $R(\lambda)$ is singular, then at least one of these null-spaces is non-trivial. If ${\cal N}_{r}(R)$ (resp. ${\cal N}_{\ell}(R)$) is non-trivial, it has minimal bases and minimal indices, which are called the right (resp. left) minimal bases and minimal indices of $R(\la)$. Notice that an $m\times n$ rational matrix of normal rank
  $r$ has $m-r$ left minimal indices and $n-r$ right minimal indices.

The {\em complete list of structural data of a rational matrix} is formed by its zero structure, its pole structure, and its left and right minimal indices.

The following degree sum theorem \cite{VVK79} relates the structural data of a rational matrix $R(\la)$. In particular, it relates the McMillan degree $\delta_p (R)$ and the zero degree $\delta_z(R)$ of $R(\la)$ to the {\em left null space degree} $\delta_\ell(R)$ of $R(\la),$ that is the sum of all left minimal indices, and to the {\em right null space degree} $\delta_r(R)$ of $R(\la),$ that is the sum of all right minimal indices.
  \begin{theorem} \label{degreesum}
  	Let $R(\la)\in  \mathbb{C}(\la)^{m\times n}$ be a rational matrix. Then
  	$$\delta_p(R)= \delta_z(R)+ \delta_\ell(R)+ \delta_r(R).$$
  \end{theorem}

\subsection{Strongly minimal linearizations and their relation with other cla\-sses of linearizations} \label{subsec.stronglyminlinear}
  Linearizing rational matrices is one of the most competitive methods for computing their complete lists of structural data. This means constructing a matrix pencil such that the complete list of structural data of the corresponding rational matrix can be recovered from the structural data of the pencil. In this paper, we focus on the strongly minimal linearizations introduced in Definition \ref{def.maquintstrongmin}. For the purpose of comparing our results with others available in the literature, we also revise very briefly other notions of linearizations.

  Since the results in this paper are relevant also when they are applied to polynomial matrices, we start with a very popular notion of linearization of a polynomial matrix. A pencil $L(\lambda)$ of degree $1$ is a linearization in the sense of Gohberg, Lancaster and Rodman \cite{GohbergLancasterRodman09}, or in the GLR-sense for short, of a polynomial matrix $P(\lambda)$ of degree $d >1$, if there exist unimodular matrices $U(\lambda)$ and $V(\lambda)$ such that
 \[
 U(\lambda)L(\lambda)V(\lambda) =
 \begin{bmatrix}
 P(\lambda) & 0 \\
 0 & I_s
 \end{bmatrix},
 \]
 where $I_s$ denotes the identity matrix of size any integer $s\geq 0$. The key property of a GLR-linearization is that it has the same finite eigenvalues with the same partial multiplicities as $P(\la)$.
 Furthermore, $L(\lambda)$ is a strong linearization of $P(\la)$ in the GLR-sense if $L(\la)$ is a GLR-linearization of $P(\la)$ and $ \text{rev}_1L(\lambda)$ is a GLR-linearization of $\text{rev}_d P(\lambda)$. Then, a GLR-strong linearization has the same finite and infinite eigenvalues with the same partial multiplicities as $P(\la)$. However, the minimal indices of a GLR (strong) linearization $L(\la)$ may be completely unrelated to those of $P(\la)$ \cite[Section 4]{DeTDM}, except for the fact that the number of left (resp. right) minimal indices of $L(\la)$ and $P(\la)$ are equal. Nevertheless, the GLR-strong linearizations that are used in practice have minimal indices that are simply related to those of the polynomial through addition of a constant shift (see \cite{DLPV18} and the references therein).

In order to linearize a rational matrix $R(\la)\in \mathbb{C}(\la)^{m \times n}$, we consider in this paper {\em linear} polynomial system matrices of $R(\la)$ \cite{Ros70}. This means that we consider block partitioned pencils
\begin{equation}\label{eq:linearsystemmat}
L(\lambda):= \left[\begin{array}{ccc} \la A_1 -A_0 & - \la B_1 + B_0 \\ \la C_1 -C_0 &  \la D_1-D_0 \end{array}\right] =: \begin{bmatrix}
  A(\la) & -B(\la)\\
  C(\la) & D(\la)
  \end{bmatrix} \in \mathbb{C}[\la]^{(p+m) \times (p+n)},
 \end{equation}
 where $A(\la)\in \mathbb{C}[\la]^{p\times p}$ is regular and the Schur complement of $A(\la)$ in $L(\la)$ is the rational matrix $R(\la)$, i.e., $R(\la) =D(\la)+C(\la)A( \la)^{-1}B(\la)$. In this situation, it is  also said that $R(\la)$ is the transfer function matrix of $L(\la).$  These pencils are particular instances of Rosenbrock's polynomial system matrices \cite{Ros70}, which may have any degree.

A linear polynomial system matrix $L(\la)$ as in \eqref{eq:linearsystemmat} contains the {\em finite} zero and pole structures of its transfer function matrix $R(\la)$ provided that $L(\la)$ satisfies the following minimality conditions. $L(\la)$ is minimal if the matrices
 \begin{equation}  \label{minimalset}
 \left[\begin{array}{ccc} \la A_1 -A_0  & - \la B_1 + B_0  \end{array}\right], \quad \left[\begin{array}{ccc} \la A_1 -A_0  \\ \la C_1 - C_0   \end{array}\right],
 \end{equation}
 have, respectively, full row and column rank for all $\la_0\in\mathbb{C}$. This is equivalent to state that the pencils in \eqref{minimalset} do not have finite eigenvalues. Then we have the following result.

  \begin{theorem}\cite{Ros70}\label{prop:finite} Let $R(\la)$ be the transfer function matrix of $L(\la)$ in \eqref{eq:linearsystemmat}. Let $\la_0\in\mathbb{C}$. If $L(\la)$ is minimal then
 		\begin{enumerate}
 		\item the zero structure of $R(\la)$ at $\la_0$ is the same as the zero structure of $L(\la)$ at $\la_0$, and
 		\item the pole structure of $R(\la)$ at $\la_0$ is the same as the zero structure of $\la A_1 -A_0$ at $\la_0$.
        \end{enumerate}
 \end{theorem}

It is very easy to prove that the number of left (resp. right) minimal indices of a minimal polynomial system matrix is equal to the number of left (resp. right) minimal indices of its transfer function matrix, though their values may be different \cite{VVK79,ADMZ2}.

\begin{remark} \label{rem:minimal-glr} \rm We can combine Theorem \ref{prop:finite} applied to a polynomial matrix $P(\la)$ and the equality of the number of the minimal indices of $L(\la)$ and $P(\la)$ with \cite[Theorem 4.1]{DeTDM} for proving that {\em any minimal linear polynomial system matrix of a polynomial matrix $P(\la)$ is always a GLR-linearization of $P(\la)$}. The reverse result is not true in general. Observe also that any minimal polynomial system matrix of a polynomial matrix $P(\la)$ must have the block $A(\la)$ in \eqref{eq:linearsystemmat} unimodular, because $P(\la)$ does not have finite poles.\end{remark}

The minimal linear polynomial system matrices of an arbitrary rational matrix $R(\la)$ are particular cases of the linearizations of $R(\la)$ defined in \cite[Definition 3.2]{ADMZ}, which were introduced with the idea of combining the concept of minimal polynomial system matrix with the extension of GLR-linearizations from polynomial to rational matrices.

Notice that Theorem \ref{prop:finite} does not provide information about the structure at infinity. The recovering of this structure requires the following concept: $L(\la)$ in \eqref{eq:linearsystemmat} is minimal at infinity \cite{DMQVD} if the matrices \begin{equation} \label{condlocalinf}  \left[\begin{array}{cc} A_1 & {\color{blue}-} B_1  \end{array}\right] \quad \mathrm{and} \quad
 \left[\begin{array}{c} A_1  \\ C_1 \end{array}\right]
 \end{equation} have, respectively, full row and column rank. This condition is equivalent to state that the pencils in \eqref{minimalset} have degree exactly $1$ and do not have eigenvalues at $\infty$. Then we have the next result that follows from \cite{Ver81} and \cite[Section 3]{DQV}.

   \begin{theorem}\label{prop:infinity} Let $R(\la)$ be the transfer function matrix of $L(\la)$ in \eqref{eq:linearsystemmat}. If $L(\la)$ is minimal at $\infty$ then
   		\begin{enumerate}
   		\item the zero structure of $R(\la)$ at infinity is the same as the zero structure of $L(\la)$ at infinity, and
   		\item the polar structure of $R(\la)$ at infinity is the same as the zero structure of the pencil
   				\begin{equation}\label{larger}
   		 \left[\begin{array}{ccc} \la A_1 -A_0 & - \la B_1 & 0  \\ \la C_1  &  \la D_1  & -I_m \\
   		0 & I_n & 0 \end{array}\right]
   				\end{equation} at infinity.
   \end{enumerate}
 \end{theorem}

The polar structure of $R(\la)$ at $\infty$ can also be recovered without considering the extended pencil in \eqref{larger}. In particular, both the zero and polar structures of $R(\la)$ at infinity can be obtained from the eigenvalue structures of the pencils $L(\la)$ and $A(\la)$ at infinity as Theorem \ref{theo:recover_infinity} shows. We emphasize that the hypothesis of minimality at $\infty$ used in Theorem \ref{theo:recover_infinity} implies that $L(\la)$ has degree $1$. However, $A(\la) = \la A_1 - A_0$ might have degree $0$ if $A_1= 0$. In any case, we understand that $\mathrm{rev}_1 A(\la) =  A_1 - \la A_0$.

\begin{theorem}\label{theo:recover_infinity} \cite[Theorem 3.13]{DMQVD} Let $R(\la)$ be the transfer function matrix of $L(\la)$ in \eqref{eq:linearsystemmat}. Assume that $R(\la)$ has normal rank $r$. Let $0 <e_1\leq\cdots\leq e_s$ be the partial multiplicities of $\mathrm{rev}_1 A(\la)$ at $0$ and let $0<\widetilde{e}_1\leq\cdots\leq\widetilde{e}_u$ be the partial multiplicities of $\mathrm{rev}_1 L(\la)$ at $0.$ If $L(\la)$ is minimal at $\infty$ then the structural indices at infinity $d_{1}\leq\cdots\leq d_{r}$ of $R(\la)$ are
\begin{equation*}
	(d_{1}, d_{2},\ldots , d_{r})=(-e_{s},-e_{s-1},\ldots ,-e_1,\underbrace{0,\ldots,0}_{r-s-u},\widetilde{e}_{1}, \widetilde{e}_{2},\ldots, \widetilde{e}_{u}) - (1,1,\ldots, 1).
	\end{equation*}
\end{theorem}

A linear polynomial system matrix that is minimal (at finite points) and also minimal at $\infty$ is called {\em strongly minimal} \cite{DMQVD,DQV}. Related to this concept we present the following definitions, which have been introduced in \cite[Section 3]{DQV} for polynomial system matrices of any degree.

\begin{definition}\label{min}
	 A linear polynomial system matrix $L(\la)$ as in \eqref{eq:linearsystemmat} is said to be {\em strongly E-controllable} and {\em strongly E-observable}, respectively, if the pencils
	\begin{equation}  \label{condlocal} \left[\begin{array}{cc} A(\la) & {\color{blue}-}B(\la)  \end{array}\right], \quad \mathrm{and} \quad
	\left[\begin{array}{c} A(\la)  \\ C(\la)  \end{array}\right],
	\end{equation}
	have degree exactly $1$ and have no finite or infinite eigenvalues. If both conditions are satisfied $L(\la)$ is said to be {\em strongly minimal}.
\end{definition}

The letter E in the definition of strong E-controllability and E-observability refers to the condition of the matrices in \eqref{condlocal} not having {\em eigenvalues}, finite or infinite, and emphasizes the differences with the concepts of ``strong controlability, observability and irreducibility'' used in \cite{VVK79,Ver81,DQV}. As mentioned before, the degree $1$ pencils in \eqref{condlocal} do not have infinite eigenvalues if and only if the matrices in \eqref{condlocalinf} have full row and full column rank, respectively. The ranks of the matrices in \eqref{condlocalinf} will be also called the {\em ranks at infinity} of the pencils in \eqref{condlocal}, even in the case the matrices in \eqref{condlocalinf} do not have full ranks.

Next, we introduce formally the definition of strongly minimal linearization of a rational matrix, which is fundamental in this work. This definition is implicit in \cite{DQV}.

\begin{definition} \label{def.maquintstrongmin}   Let $R(\la)\in\BC(\la)^{m\times n}$ be a rational matrix. A linear polynomial system matrix $L(\la)$ as in \eqref{eq:linearsystemmat} is said to be {\em a strongly minimal linearization of $R(\la)$} if $L(\la)$ is strongly minimal and its transfer function matrix is $R(\la)$. Equivalently, $\{A(\la), B(\la), C(\la), D(\la)\}$ is said to be {\em a strongly minimal linear realization} of $R(\la)$.\end{definition}

Strongly minimal linearizations $L(\la)$ of a rational matrix $R(\la)$ have been defined with the goal of constructing pencils that allow us to recover the complete pole and zero structures of $R(\la)$ through Theorems \ref{prop:finite} and \ref{theo:recover_infinity}, or \ref{prop:infinity}. Surprisingly, the condition of strong minimality implies that the minimal indices of $L(\la)$ and $R(\la)$ are the same. This is proved in Theorem \ref{prop:minimalindices}, which, together with Theorems \ref{prop:finite} and \ref{theo:recover_infinity}, {\em allows us to recover the complete list of structural data of a rational matrix from any of its strongly minimal linearizations}.

 \begin{theorem}\label{prop:minimalindices}  Let $L(\la)$ be a  strongly minimal linearization of a rational matrix $R(\la)$. Then the left and right minimal indices of $R(\la)$ are the same as the left and right minimal indices of $L(\la)$.
 \end{theorem}

\begin{proof}
By \cite[Proposition 1]{DQV}, a strongly minimal linear polynomial system matrix is strongly irreducible according to the definition in \cite{Ver81}. Then, by \cite[Result 2]{Ver81}, the left and right minimal indices of $R(\la)$ and $L(\la)$ are the same.
\end{proof}

As we have seen in the proof of Theorem \ref{prop:minimalindices}, any strongly minimal linearization $L(\la)$ of a rational matrix $R(\la) \in \mathbb{C}(\la)^{m\times n}$ is a strongly irreducible polynomial system matrix of $R(\la)$ (see definition in \cite{Ver81}). Thus, \cite[Result 2]{Ver81} establishes a simple bijection between the left (resp. right) minimal bases of $L(\la)$ and those of $R(\la)$ that allows us to recover a left (resp. right) minimal basis of $R(\la)$ from any left (resp. right) minimal basis of $L(\la)$, and conversely, without any computational cost. We only state here the result for right minimal bases since for left minimal bases is analogous.

	\begin{theorem} \label{prop:basesmin} Let $L(\la)$ as in \eqref{eq:linearsystemmat} be a strongly minimal linearization of a rational matrix $R(\la)$. If the columns of $\begin{bmatrix}
	M_{1}(\la)\\ M_{2}(\la) \end{bmatrix}$, partitioned conformably to the blocks of $L(\la)$, form a right minimal basis for $L(\la)$ then the columns of $M_2(\la)$ form a right minimal basis for $R(\la)$. Conversely, if the columns of $M_2(\la)$ form a right minimal basis for $R(\la)$ then the columns of $\begin{bmatrix}
	A(\la)^{-1}B(\la)M_{2}(\la)\\ M_{2}(\la) \end{bmatrix}$ form a right minimal basis for $L(\la)$.
\end{theorem}

\begin{remark} \rm Given $\la_0 \in \mathbb{C}$ with $\det A(\la_0)\neq 0$, it is easy to prove that the same recovery rules of Theorem \ref{prop:basesmin} hold for the bases of the left (resp. right) null space of the constant matrix $R(\la_0)$ from those of the left (resp. right) null space of the constant matrix $L(\la_0)$ for any linear polynomial system matrix $L(\la)$ as in \eqref{eq:linearsystemmat}, without imposing strong minimality (see \cite[Section 5.1]{DMQ}).  In the case of regular rational matrices, $\la_0\in \mathbb{C}$ is an eigenvalue of $R(\la)$ when it is a zero but not a pole, and the finite poles of $R(\la)$ are the finite zeros of $A(\la)$ if $L(\la)$ is minimal. Then, by assuming minimality on $L(\la)$, the previous rule allows us to recover the associated eigenvectors of $R(\la)$ from those of $L(\la)$.
\end{remark}

\begin{remark}
It follows from Theorems \ref{prop:finite}, \ref{prop:infinity} and \ref{prop:minimalindices} that, if $L(\la)$ is a strongly minimal linearization of a rational matrix $R(\la)$, then
	$$\delta_z(R)+\delta_\ell(R)+ \delta_r(R)=\delta_z(L)+\delta_\ell(L)+ \delta_r(L),$$
	and then from Theorem  \ref{degreesum} that $\delta_p(R)=\delta_p(L).$
	But the only pole of $L(\la):=\la L_{1} + L_0$ is the point at infinity and its polar degree is equal to $\rank{L_1}$ \cite[p. 126]{Van81}. Therefore, the McMillan degree $\delta_p(R)$ of $R(\la)$ equals the rank of $L_1$ for any strongly minimal linearization of $R(\la)$ and no other pencils with the same zero structure and the same left and right minimal indices as $R(\la)$ can have a first order coefficient with smaller rank. Thus, strongly minimal linearizations are optimal in this sense.
\end{remark}

By Remark \ref{rem:minimal-glr}, we have that strongly minimal linearizations of a {\em polynomial matrix} $P(\la)$ are always GLR-linearizations of $P(\la)$. However, the following example shows that they are not, in general, GLR-strong linearizations.

\begin{example} \label{ex.minimalvsGLR} \rm (Strongly minimal linearizations of polynomial matrices are not strong linearizations in the sense of Gohberg, Lancaster and Rodman) Consider the polynomial matrix
\[
P (\la) =
\la^2 \begin{bmatrix}
        0 & 0 \\
        0 & 1
      \end{bmatrix}
+ \la  \begin{bmatrix}
        1 & 0 \\
        0 & 1
      \end{bmatrix}
+ \begin{bmatrix}
        1 & 0 \\
        0 & 1
      \end{bmatrix}
\]
and the partitioned pencil
\[
L(\la) =
\left[ \begin{array}{c|cc}
         -1 & 0 & \la \\ \hline \phantom{\Big|}
         0 & \la + 1 & 0 \\
         \la & 0 & \la + 1
       \end{array}
\right].
\]
The transfer function matrix of $L(\la)$ is $P(\la)$ and $L(\la)$ is minimal and minimal at infinity. Therefore, $L(\la)$ is a strongly minimal linearization of $P(\la)$ and also a GLR-linearization of $P(\la)$. However, $\mathrm{rev}_1 L(\la)$ is {\em not} unimodularly equivalent to $\mathrm{diag} (\mathrm{rev}_2 P(\la), 1)$ and, thus, $L(\la)$ is not a GLR-strong linearization of $P(\la)$. In order to see this, observe that
\[
\mathrm{rev}_2 P (\la) =
\la^2 \begin{bmatrix}
        1 & 0 \\
        0 & 1
      \end{bmatrix}
+ \la  \begin{bmatrix}
        1 & 0 \\
        0 & 1
      \end{bmatrix}
+ \begin{bmatrix}
        0 & 0 \\
        0 & 1
      \end{bmatrix} \;\; \mbox{and} \;\;
\mathrm{rev}_1 L (\la) =
\left[ \begin{array}{c|cc}
         -\la & 0 & 1 \\ \hline \phantom{\Big|}
         0 & \la + 1 & 0 \\
         1 & 0 & \la + 1
       \end{array}
\right],
\]
which makes it transparent that $\mathrm{rev}_1 L(\la)$ does not have eigenvalues (or zeros) at zero, while $\mathrm{rev}_2 P(\la)$ does. In general, it is possible to prove by using Theorem \ref{theo:recover_infinity} that strongly minimal linearizations of polynomial matrices of degree larger than $1$ with eigenvalues at infinity are not GLR-strong linearizations.
\end{example}

Despite of the fact of not being GLR-strong linearizations, strongly minimal linearizations of a polynomial matrix $P(\la)$ allow us to recover always the complete list of structural data of $P(\la)$, including its minimal indices. Moreover, we will prove in this paper that they allow us to preserve structures of polynomial matrices that cannot always be preserved by GLR-strong linearizations.

Finally, note that according to the definitions in \cite{DMQVD}, we can also say that a strongly minimal linearization $L(\la)$ of a rational matrix $R(\la)$ is a linearization of its transfer function matrix $R(\la)$ at all finite points and also at infinity. However, strongly minimal linearizations are not always strong linearizations in the sense of \cite[Definition 3.4]{ADMZ} since the first degree coefficients of their $(1,1)$-blocks are not necessarily invertible.


\section{Constructing strongly minimal linearizations of polynomial matrices}\label{sec:linearization_polynomial}

In this section we focus on constructing explicitly a strongly minimal linearization for any given polynomial matrix $P(\la)\in \BC[\la]^{m\times n}$ of degree $d>1$~:
\begin{equation}\label{poly}
P(\la):= P_0 + P_1\la + \cdots + P_d\la^d.
 \end{equation}
Such a strongly minimal linearization is constructed in Theorem \ref{deflate} and we will prove in Section \ref{sec:self-cong-poly} that it inherits the structure of $P(\la)$, when $P(\la)$ possesses any of the self-conjugate structures considered in this work. The construction uses three pencils associated with $P(\la)$ that have appeared before in the literature. They are described in the following paragraphs.

The pencil
\begin{equation}\label{companion}
L_r(\la):=\left[\begin{array}{c|c} A_r(\la) & -B_r(\la) \\ \hline \phantom{\Big|} C_r(\la) & D_r(\la) \end{array}\right] :=
 \left[\begin{array}{cccc|c} -I_n & \la I_n & & &  0 \\  & \ddots & \ddots & & \vdots \\ & & -I_n  & \la I_n & 0 \\
 & & & -I_n & \la I_n \\ \hline \phantom{\Big|}   \la P_d & \ldots & \ldots & \la P_2 & \la P_1+P_0
  \end{array}\right]
\end{equation}
was used in the classical reference \cite{VanD83}.
It is easy to see that $L_r(\la)$ is a linear polynomial system matrix of $P(\la)$, since $P(\la) = D_r(\la) + C_r (\la) A_r (\la)^{-1} B_r (\la)$, and that it
is minimal for all finite $\la$. For the point at  $\infty$, E-controllability is clearly satisfied but E-observability is only satisfied if the matrix $P_d$ has full column rank $n$. Thus, $L_r (\la)$ is {\em not} a strongly minimal linearization of $P(\la)$ when $P_d$ does not have full column rank. However, note that $L_r (\la)$ is always a GLR-strong linearization of $P(\la)$. This can be seen, for instance, by noting that if the two block rows in \eqref{companion} are interchanged, we obtain one of the block Kronecker linearizations (with only one block column) associated to $P(\la)$ defined in \cite[Section 4]{DLPV18}. The pencil \eqref{companion} has a structure similar to that of the classical first or row Frobenius companion form.

The pencil
\begin{equation}\label{dual}
L_c(\la):=\left[\begin{array}{c|c} A_c(\la) & -B_c(\la) \\ \hline \phantom{\Big|} C_c(\la) & D_c(\la) \end{array}\right] :=
 \left[\begin{array}{cccc|c} -I_m & & & &  \la P_d \\  \la I_m & \ddots & & & \vdots \\ &  \ddots & -I_m  & & \vdots \\
 & & \la I_m  & -I_m & \la P_2 \\ \hline \phantom{\Big|} 0  & \ldots & 0 & \la I_m & \la P_1+P_0
  \end{array}\right]
\end{equation}
is in some sense ``dual'' to \eqref{companion}. It is also a linear polynomial system matrix of $P(\la)$, since $P(\la) = D_c(\la) + C_c (\la) A_c (\la)^{-1} B_c (\la)$. Moreover, $L_c (\la)$ is strongly E-observable, but not necessarily  strongly E-controllable, unless $P_d$ has full row rank. As a consequence, $L_c (\la)$ is a strongly minimal linearization of $P(\la)$ if and only if $P_d$ has full row rank. However, $L_c (\la)$ is always a GLR-strong linearization of $P(\la)$. The pencil \eqref{dual} has a structure similar to that of the classical second or column Frobenius companion form.

The pencil {\small
\begin{equation}\label{lancaster}
L_s(\la):=\left[\begin{array}{c|c} A_s(\la) & -B_s(\la) \\ \hline \phantom{\Big|} C_s(\la) & D_s(\la) \end{array}\right] :=
 \left[\begin{array}{cccc|c} & & & -P_d &  \la P_d  \\  & & \adots &  \la P_d -P_{d-1} & \vdots \\  & -P_d & \adots & \vdots  & \vdots\\ [+2mm]
  -P_d &  \la P_d -P_{d-1} & \ldots & \la P_3 -P_2 & \la P_2 \\ \hline \phantom{\Big|}   \la P_d & \ldots & \ldots & \la P_2 & \la P_1+P_0
  \end{array}\right]
\end{equation}}
was originally proposed by Lancaster in \cite[pp. 58-59]{Lancaster66} for regular polynomial matrices with $P_d$ invertible. In this paper, we use it for arbitrary polynomial matrices, including rectangular ones. $L_s (\la)$ has the advantage to preserve the Hermitian or skew-Hermitian nature of the coefficients of the linearization, if $P(\la)$ happens to have coefficients with such properties. The pencil $L_s (\la)$ has been also studied more recently in \cite{hmmt2006,mmmm2006}, where it is seen as one of the pencils of the standard basis of the linear space $\mathbb{DL} (P)$ of pencils related to $P(\la)$. It is well known that $L_s (\la)$ is a GLR-strong linearization of $P(\la)$ if and only if $P_d$ is invertible \cite{DeTDM2009,mmmm2006}. In fact, in this case, $L_s (\la)$ is also a strongly minimal linearization of $P(\la)$ since it is strongly minimal and $P(\la) = D_s(\la) + C_s (\la) A_s (\la)^{-1} B_s (\la)$. However, if $P_d$ is not invertible, $L_s (\la)$ is not a linearization of $P(\la)$ in any of the senses considered in the literature and, even more, it is not a Rosenbrock polynomial system matrix of $P(\la)$ since $A_s (\la)$ is not regular. Despite of this fact, $L_s (\la)$ is our starting point for constructing the strongly minimal linearization of $P(\la)$ of interest in this work.

The constant block Hankel matrix $T$ defined in the next equation
\begin{equation} \label{toeplitz}
T:= \left[\begin{array}{cccc} & & & P_d \\  & & \adots &  P_{d-1} \\  & P_d & \adots & \vdots \\
P_d &  P_{d-1} & \ldots & P_2
\end{array}\right]
\end{equation}
plays a key role in the rest of the paper. To begin with, it allows us
to obtain the following relations
\begin{equation} \label{leftright}
\left[\begin{array}{cc} A_s(\la) & -B_s(\la)  \end{array}\right] =T \, \left[\begin{array}{cc} A_r(\la) & -B_r(\la) \end{array}\right],
\quad
\left[\begin{array}{cc} A_s(\la) \\ C_s(\la)  \end{array}\right] =\left[\begin{array}{cc} A_c(\la) \\ C_c(\la) \end{array}\right]\, T,
\end{equation}
between submatrices of the pencils $L_s (\la)$, $L_r (\la)$ and $L_c (\la)$. The matrix $T$ is invertible if and only if $P_d$ is square and invertible. Otherwise, $T$ is singular and this is the case that requires a careful analysis.

In \cite{VanD83}, it was shown how to derive from the linear polynomial system matrix $L_r(\la)$ of $P(\la)$, a smaller linear polynomial system matrix $\widehat L_r(\la)$ that is both strongly E-controllable and E-observable, and hence strongly minimal, by using only multiplications by constant unitary matrices. This was obtained by deflating the unobservable infinite eigenvalues from the pencil $L_r(\la)$. Moreover, the obtained pencil $\widehat L_r(\la)$ allows us to recover the complete list of structural data of $P(\la)$. The reduction procedure in \cite{VanD83} has been recently extended to arbitrary linear polynomial system matrices of arbitrary rational matrices $R(\la)$ in \cite{DQV}, where it is proved that the obtained strongly minimal linear polynomial system matrix has as transfer function matrix $Q_1 R(\la) Q_2$, where $Q_1$ and $Q_2$ are constant invertible matrices. We emphasize that the procedures in \cite{VanD83,DQV} lead to stable and efficient numerical algorithms since both are based on unitary transformations.

We show in Theorem \ref{deflate} that a procedure similar to that in \cite{VanD83} can be applied to $L_s(\la)$ in order to derive a strongly minimal linear polynomial system matrix $\widehat L_s(\la)$ of $P(\la)$, despite of the fact that if $P_d$ is not square or invertible, then {\em $L_s(\la)$ is not a Rosenbrock polynomial system matrix} since $A_s(\la)$ is then not regular. Moreover, we remark that the procedure in Theorem \ref{deflate} is much simpler than those in \cite{VanD83,DQV} and that, as said before, it yields a polynomial system matrix whose transfer function matrix is precisely $P(\la)$. Before stating and proving Theorem \ref{deflate}, we prove the simple auxiliary Lemma \ref{transfer} and introduce some other auxiliary concepts.

A rational matrix $G(\la)\in \mathbb{C} (\la)^{p\times n}$ (with $p<n$) is said to be a rational basis if its rows form a basis of the rational subspace they span, i.e., if it has full row normal rank. Two rational bases $G(\la) \in \mathbb{C}(\la)^{p\times n}$ and $H(\la) \in \mathbb{C} (\la)^{q\times n}$ are said to be dual if $p+q = n$, and $G(\la) \, H(\la)^T =0$.
\begin{lemma}\label{transfer}
	Let $$S(\lambda):=\left[\begin{array}{ccc} A(\la) & -B(\la) \\ C(\la) & D(\la) \end{array}\right] \in \mathbb{C}[\la]^{(p+m) \times (p+n)}$$ be a polynomial system matrix, where $A(\la)$ is assumed to be regular.  Let $H(\la)$ be a rational basis of the form  $H(\la):=\left[\begin{array}{cc}M(\la) & I_n\end{array}\right]$ dual to $\left[\begin{array}{cc}A(\la) & -B(\la)\end{array}\right],$ i.e., such that $\left[\begin{array}{cc}A(\la) & -B(\la)\end{array}\right]H(\la)^T=0,$ then $\left[\begin{array}{cc}C(\la) & D(\la)\end{array}\right]H(\la)^{T}$ is the transfer function of $S(\la)$.
\end{lemma}
\begin{proof}
	The equation
	$$\left[\begin{array}{cc}  A(\la) & - B(\la) \end{array}\right]
	\left[\begin{array}{cc}  M(\la)^{T} \\ I_n \end{array}\right]=  0
	$$
	implies $A(\la) M(\la)^{T}= B(\la) $ and, since $A(\la)$ is regular, $M(\la)^{T}=A(\la)^{-1}B(\la) .$ Thus $\left[\begin{array}{cc}C(\la) & D(\la)\end{array}\right]H(\la)^{T}=C(\la)A(\la)^{-1}B(\la)+D(\la)$.
\end{proof}

\begin{theorem} \label{deflate}
Let $P(\la)\in \BC[\la]^{m\times n}$ be a polynomial matrix as in \eqref{poly}. Let $T$ be the block Hankel matrix in \eqref{toeplitz} and $r:=\rank T$. Let $U=\left[\begin{array}{cc} U_1 & U_2  \end{array}\right]$ and $V=\left[\begin{array}{cc} V_1 & V_2  \end{array}\right]$ be unitary matrices that ``compress'' the matrix $T$ as follows~:
\begin{equation} \label{eq.compressT} U^*TV =\left[\begin{array}{cc} 0 &  0 \\ 0 & U_2^*TV_2 \end{array}\right]=: \left[\begin{array}{cc} 0 &  0 \\ 0 & \widehat T \end{array}\right],
\end{equation}
where $\widehat T$ is of dimension $r\times r$ and invertible. Then, if $L_s(\la)$ is the matrix pencil in \eqref{lancaster}, the pencil $\mathrm{diag} (U^*, I_m)  \,  L_s(\la) \, \mathrm{diag} (V, I_n)$ is equal to the ``compressed'' pencil
\begin{equation} \label{compressed}
\left[\begin{array}{cc|c} 0 & 0 & 0 \\ 0 & \widehat A_s(\la) & -\widehat B_s(\la) \\ \hline \phantom{\Big|} 0 & \widehat C_s(\la) & \widehat D_s(\la) \end{array}\right]:=\left[\begin{array}{c|c} U^*A_s(\la)V & -U^*B_s(\la) \\ \hline \phantom{\Big|} C_s(\la)V & D_s(\la) \end{array}\right] ,
\end{equation}
 and
	\begin{equation} \label{eq.widehatL_s1} \widehat L_s(\la):=\left[\begin{array}{c|c} \widehat A_s(\la) & -\widehat B_s(\la) \\  \hline \phantom{\Big|} \widehat C_s(\la) & \widehat D_s(\la) \end{array}\right]
	\end{equation}
 is a strongly minimal linearization of $P(\la)$, where $\widehat A_s(\la) \in \mathbb{C}[\la]^{r \times r}$ is regular. In particular, $P(\la)=\widehat D_s(\la)+\widehat C_s(\la)\widehat A_s(\la)^{-1}\widehat B_s(\la)$.
\end{theorem}

\begin{proof}
It follows from \eqref{leftright} and the strong E-controllability of $\left[\begin{smallmatrix} A_r(\la) & -B_r(\la)  \end{smallmatrix}\right]$
that  $\left[\begin{smallmatrix} A_s(\la) & -B_s(\la)  \end{smallmatrix}\right]$ has rank $r$ for all $\la$, infinity included, and that its
left null space is spanned by the rows of $U_1^*$. Likewise, it follows from \eqref{leftright} and the strong E-observability of $\left[\begin{smallmatrix} A_c(\la) \\ C_c(\la)  \end{smallmatrix}\right]$  that  $\left[\begin{smallmatrix} A_s(\la) \\ C_s(\la)  \end{smallmatrix}\right]$ has rank $r$ for all $\la$, infinity included and that its right null space is spanned by the columns of $V_1$.
This proves the compressed form \eqref{compressed}.

We then prove that the $r\times r$ matrix pencil $\widehat A_s(\la)$ is regular. This follows from the identity
\begin{equation} \label{eq.widehatAs} \widehat A_s(\la) = U_2^*T A_r(\la) V_2, \quad \mathrm{where} \quad  A_r(\la) = \left[\begin{array}{cccc} -I_n & \la I_n & &  \\  & \ddots & \ddots &  \\ & & -I_n  & \la I_n  \\
 & & & -I_n    \end{array}\right]
\end{equation}
which, for $\la=0$ becomes $ \widehat A_s(0) = -U_2^*TV_2=-\widehat T$.

The fact that $\left[\begin{smallmatrix} \widehat A_s(\la) & - \widehat B_s(\la)  \end{smallmatrix}\right]$ has full row rank $r$ for all $\la$, $\infty$ included, follows from the identity
$$ \left[\begin{array}{ccc} 0 & \widehat A_s(\la) & - \widehat B_s(\la)  \end{array}\right] = \widehat T V^*_2 \left[\begin{array}{cc} A_r(\la) & -B_r(\la)  \end{array}\right] \diag{V,I_n}.
$$
The fact that $\left[\begin{smallmatrix} \widehat A_s(\la) \\ \widehat C_s(\la)  \end{smallmatrix}\right]$ has full column rank $r$ for all $\la$, $\infty$ included,
follows from the dual identity
$$ \left[\begin{array}{ccc} 0 \\ \widehat A_s(\la) \\  \widehat C_s(\la)  \end{array}\right] = \diag{U^*, I_m} \left[\begin{array}{cc} A_c(\la) \\ C_c(\la)  \end{array}\right] U_2 \widehat T.
$$

Together, these properties guarantee that  $ \widehat L_s(\la)$ is a strongly minimal linear polynomial system matrix. Its transfer function $\widehat C_s(\la)\widehat A_s(\la)^{-1}\widehat B_s(\la)+ \widehat D_s(\la) $ can then be obtained from a particular dual basis $N(\la)\in \BC(\la)^{ n\times (r+n) }$ of $\left[\begin{smallmatrix} \widehat A_s(\la) & - \widehat B_s(\la)  \end{smallmatrix}\right]$, by using Lemma \ref{transfer}. Since
\begin{align*}
\left[\begin{array}{cc|c} 0 & 0 & 0 \\ 0 & \widehat A_s(\la) & - \widehat B_s(\la)  \end{array}\right]\diag{V^*,I_n}
&=U^* \left[\begin{array}{cc} A_s(\la) & -B_s(\la)  \end{array}\right] \\ & = U^* T \left[\begin{array}{cc} A_r(\la) & -B_r(\la)  \end{array}\right],
\end{align*}
it follows that
$$ \left[\begin{array}{cc|c} 0 & 0 & 0 \\ 0 & \widehat A_s(\la) & - \widehat B_s(\la)  \end{array}\right]\diag{V^*,I_n}   \left[\begin{array}{c} \la^{d-1} I_n \\ \vdots \\ \la I_n \\ I_n   \end{array}\right]=0
$$
and hence that
$$  \left[\begin{array}{c|c} \widehat A_s(\la) & - \widehat B_s(\la)  \end{array}\right]\diag{\left[\begin{array}{cc} 0 & I_r \end{array}\right]V^*,I_n}   \left[\begin{array}{c} \la^{d-1} I_n \\ \vdots \\ \la I_n \\ I_n   \end{array}\right]=0.
$$
Therefore, by setting
$$  N(\la)^{T}:=  \diag{\left[\begin{array}{cc} 0 & I_r \end{array}\right]V^*,I_n} \left[\begin{array}{c} \la^{d-1} I_n \\ \vdots \\ \la I_n \\ I_n   \end{array}\right]  \in \mathbb{C}[\la]^{(r+n) \times n},
$$
we have that $N(\la)$ is a dual basis of $ \left[\begin{smallmatrix} \widehat A_s(\la) & - \widehat B_s(\la)  \end{smallmatrix}\right]$ with its rightmost block equal to $I_n$. By Lemma \ref{transfer}, and using the fact that $$  \widehat C_s(\la) \left[\begin{array}{cc} 0 & I_r \end{array}\right]V^*=\left[\begin{array}{cc} 0 & \widehat C_s(\la) \end{array}\right]V^*=C_s(\la),$$
we obtain that
$$  \left[\begin{array}{cc} \widehat C_s(\la) & \widehat D_s(\la) \end{array}\right] N(\la)^{T}= \left[\begin{array}{cc}  C_s(\la) &  D_s(\la) \end{array}\right]  \left[\begin{array}{c} \la^{d-1} I_n \\ \vdots \\ \la I_n \\ I_n   \end{array}\right]= P(\la)$$
is the transfer function of  $\widehat L_s(\la)$.
\end{proof}

\begin{remark} \label{rem.computmaintheorem} Once the unitary matrices $U$ and $V$ (or just their last $r$ columns, $U_2$ and $V_2$) and the matrix $\widehat{T}$ in \eqref{eq.compressT} are computed, Theorem \ref{deflate} yields an efficient and stable algorithm for computing the strongly minimal linear realization $\{\widehat{A}_s (\la), \widehat{B}_s (\la), \widehat{C}_s (\la) , \widehat{D}_s (\la) \}$ of $P(\la)$. An expensive method for computing $U$, $V$ and $\widehat{T}$ is to compute the SVD of $T$, in economic version if just $U_2$ and $V_2$ are required. A cheaper method is to use the complete orthogonal decomposition in \cite[Sec. 5.4.2]{golub-van-loan-3}, which amounts to compute two QR factorizations. The block Hankel structure of $T$ (which by flipping the order of the block rows becomes block Toeplitz) allows us to use the very fast and stable method in \cite[Sec. IV]{vandooren-laurent-1979}, which makes all the computations on $m\times n$ submatrices. The method in \cite{vandooren-laurent-1979} has the additional advantage that if $r_d = \mathrm{rank} (P_d)$, then the rows of $U^* (1:(m-r_d) , 1:m)$ and the columns of $V(1:n, 1 :(n-r_d))$ (where we used MATLAB's notation) of the computed $U$ and $V$ are, respectively, unitary bases of the left and right nullspaces of $P_d$. Recall that these subspaces are precisely the left and right eigenspaces associated to the infinite eigenvalue of $P(\la)$ when $P(\la)$ is regular.
\end{remark}

\begin{remark} \label{rem.L_strans}  Even though the pencil $L_s(\la)$ in \eqref{lancaster} is not a Rosenbrock polynomial system matrix and neither is a GLR-linearization of $P(\la)$ if $P_d$ is rectangular or square and singular, it is easy to see, by using unimodular transformations that are well-known in the literature, that it has the same finite eigenvalues as $P(\la)$ with the same partial multiplicities. For this purpose, note that
$V(\la) L_s(\la) W(\la)= \diag{-T,P(\la)}$, where
$$
V(\la):=
\left[\begin{array}{ccccc}
1 & & & & \\
0 & 1 & & & \\
\vdots & \ddots & \ddots & & \\
0 & \cdots & 0 & 1 &  \\
\la^{d-1} & \cdots & \la^2 & \la & 1
\end{array}\right]\otimes I_m,
\;
W(\la):=
 \left[\begin{array}{ccccc} 1& \la  & \la^2 & \ldots & \la^{d-1}\\  & 1 &  \la & \ddots & \vdots \\ & & \ddots & \ddots &  \la ^2 \\ & & & 1  & \la \\
 & & & &  1
  \end{array}\right]\otimes I_n.
$$
Since the polynomial matrices $V(\la)$ and $W(\la)$ are unimodular, and $\diag{-T,P(\la)}$ is strictly equivalent to $\diag{0,I_r,P(\la)}$, this implies that $L_s(\la)$ is unimodularly equivalent to $\diag{0,I_r,P(\la)}$. Therefore, $L_s(\la)$ and $P(\la)$ have the same finite eigenvalues with the same partial multiplicities. Of course, this also follows from Theorem \ref{deflate} and the properties of strongly minimal linearizations studied in Section \ref{subsec.stronglyminlinear}. However, note that $L_s(\la)$ is {\em not} a GLR-linearization of $P(\la)$ because $L_s(\la)$ is {\em not} unimodularly equivalent to $\diag{I,P(\la)}$.

On the other hand, if we consider the pencil $\widehat L_s(\la)$ in \eqref{eq.widehatL_s1}, then Theorem \ref{deflate} proves that $L_s(\la)$ is strictly equivalent to $\diag{0,\widehat L_s(\la)}$. Combining the results above, we see that $\widehat L_s(\la)$ and $\diag{I_r,P(\la)}$ have the same normal rank and the same finite eigenvalues and partial multiplicities, which implies that $\widehat L_s(\la)$ and $\diag{I_r,P(\la)}$ are unimodularly equivalent \cite{Gan59} and, therefore, that $\widehat L_s(\la)$ is a GLR-linearization of $P(\la)$. This is a particular instance of the result mentioned in Section \ref{subsec.stronglyminlinear} that any minimal linear polynomial system matrix of $P(\la)$ is a GLR-linearization of $P(\la)$.
\end{remark}

\begin{remark} Observe that whenever $P_d$ is singular (which may happen even if $P(\la)$ is regular), the pencil $L_s (\la)$ is singular, since it is strictly equivalent to $\diag{0,\widehat L_s(\la)}$ by Theorem \ref{deflate}. Thus, $L_s(\la)$ has $m (d-1) -r$ left minimal indices equal to $0$ and $n (d-1) -r$ right minimal indices equal to $0$ in addition to the minimal indices of $\widehat{L}_s (\la)$ (which are precisely the minimal indices of $P(\la)$ due to Theorem \ref{prop:minimalindices}). Then, the process described in Theorem \ref{deflate} can be seen as a process that deflates from $L_s(\la)$ these additional minimal indices equal to zero in order to get a smaller pencil $\widehat L_s(\la)$ which is a strongly minimal linearization of $P(\la)$.
\end{remark}

\subsection{Quadratic polynomial matrices with low rank leading coefficient} \label{subsec.lowrank} In this subsection, we particularize the results in Theorem \ref{deflate} to quadratic polynomial matrices $$P(\la) = P_0 + \la P_1 + \la^2 P_2\in\mathbb{C}[\la]^{m\times n}$$ since they are particularly important in applications
\cite{tisseur-meerbergen}. Moreover, in some important applications, the leading coefficient $P_2$ of $P(\la)$ has low rank \cite{NLEVP}, a property related in the regular case to the presence of an infinite eigenvalue with high multiplicity. This presence is a challenge for the best available algorithms that compute all the eigenvalues of a regular dense quadratic polynomial matrix \cite{drmac,hammarling-munro-al}, since infinite eigenvalues have to be deflated. The algorithms in \cite{drmac,hammarling-munro-al} are based on the GLR-strong linearization known as the second Frobenius companion form,  but the strongly minimal linearization constructed in Theorem \ref{deflate} might be a competivite option since, as we show in this section, it is very simple and much smaller than the Frobenius companion form for quadratic polynomial matrices with $P_2$ having low rank.

Note that in the quadratic case
$$
L_s (\la) = \left[\begin{array}{c|c} -P_2 & \la P_2 \\  \hline \phantom{\Big|} \la P_2 & \la P_1 + P_0 \end{array}\right] \in \mathbb{C} [\la]^{2m \times 2n} \quad \mbox{and} \quad
T = P_2.
$$
In order to avoid confusions, in this section we denote the rank of the constant matrix $P_2$ by $r_2$ and the normal rank of $P(\la)$ by $r_P$, which in the regular case satisfies $r_P = m = n$. The rank compresion in \eqref{eq.compressT} reduces to obtaining a low rank representation of $P_2$, i.e., $P_2 = U_2 \widehat T V_2^*$, where $\widehat T \in \mathbb{C}^{r_2 \times r_2}$ is invertible and $U_2 \in \mathbb{C}^{m \times r_2}, V_2 \in \mathbb{C}^{n \times r_2}$ have orthornormal columns. Such representation can be computed via an economic SVD of $P_2$ or a complete orthogonal decomposition, as commented in Remark \ref{rem.computmaintheorem}. Then, we immediately obtain the following result either via a trivial computation or as a corollary of Theorem \ref{deflate}.

\begin{theorem} \label{thm.quadraticmain} Let $P(\la) = P_0 + P_1 \la + P_2 \la^2 \in \mathbb{C} [\la]^{m \times n}$ and $P_2 = U_2 \widehat T V_2^*$, where $\widehat T \in \mathbb{C}^{r_2 \times r_2}$ is invertible and $U_2 \in \mathbb{C}^{m \times r_2}, V_2 \in \mathbb{C}^{n \times r_2}$ have orthornormal columns. Then
$$
\widehat L_s(\la) = \left[\begin{array}{c|c} -\widehat T & \la \widehat T V_2^* \\  \hline \phantom{\Big|} \la U_2 \widehat T  & \la P_1 + P_0 \end{array}\right] \in \mathbb{C} [\la]^{(r_2 + m) \times (r_2 + n)}
$$
is a strongly minimal linearization of $P(\la)$.
\end{theorem}

Observe that if $P(\la)\in \mathbb{C} [\la]^{m \times m}$ is regular, then any GLR-strong linearization of $P(\la)$ has size $2m \times 2m$, while the size of $\widehat L_s(\la)$ is $(r_2 +m) \times (r_2 +m)$ and, so, much smaller if $r_2 \ll m$.

Another interesting feature of the quadratic polinomial case is that the recovery of the eigenvalue structure at infinity of $P(\la)$ from the one of $\widehat L_s(\la)$ is much simpler than in Theorem \ref{theo:recover_infinity}, as the next corollary shows.

\begin{corollary} \label{cor.quadraticinfty} With the same notation and hypotheses as in Theorem \ref{thm.quadraticmain}, let $r_P$ be the normal rank of $P(\la)$ and $0<\widetilde{e}_1\leq\cdots\leq\widetilde{e}_u$ be the partial multiplicities of $\mathrm{rev}_1 \widehat L_s(\la)$ at $0.$ Then, the partial multiplicities of
$\mathrm{rev}_2 P (\la)$ at $0$ are
\begin{equation*}
\underbrace{1=\cdots=1}_{r_P-r_2-u} \leq \widetilde{e}_{1} + 1 \leq \cdots \leq \widetilde{e}_{u} + 1.
\end{equation*}
\end{corollary}

\begin{proof} Observe that the partial multiplicities of $\mathrm{rev}_1 (-\widehat{T}) = -\la \widehat{T}$ at $0$ are $\underbrace{1,1,\cdots ,1}_{r_2}$ because $\widehat{T}$ is invertible. Then, the result follows from combining Theorem \ref{theo:recover_infinity} with equation \eqref{eq.infstruct22} when $d=2$.
\end{proof}

Finally, we remark that if we are given a low rank representation $P_2 = L U^*$, where $L \in \mathbb{C}^{m \times r_2}, U \in \mathbb{C}^{n \times r_2}$ have both rank $r_2$ but their columns are not necessarily orthonormal, then a direct computation shows that
$$
\left[\begin{array}{c|c} -I_{r_2} & \la U^* \\  \hline \phantom{\Big|} \la L & \la P_1 + P_0 \end{array}\right] \in \mathbb{C} [\la]^{(r_2 + m) \times (r_2 + n)}
$$
is a strongly minimal linearization of $P(\la) = P_0 + \la P_1 + \la^2 P_2$ and that Corollary \ref{cor.quadraticinfty} also holds for this linearization.

\section{Constructing strongly minimal linearizations of self-conjugate polynomial matrices} \label{sec:self-cong-poly}

The main purpose of this section is to prove that the strongly minimal linearization \eqref{eq.widehatL_s1} of the matrix polynomial $P(\la)$ developed in Theorem \ref{deflate} inherits the structure of $P(\la)$ for any of the four structures considered in this work. In the last part of this section, we compare this result with those available in the literature for structure preserving GLR-strong linearizations of structured polynomial matrices.

 We start by reviewing the four structures of interest in this paper and their properties. In these definitions, note that if $P(\la)$ is the complex polynomial matrix in \eqref{poly}, then $$[P(\la)]^* := P_0^* + P_1^* \overline{\la} + \cdots + P_d^* \overline{\la}^d.$$ The considered self-conjugate structures are:
 \begin{enumerate}
 \item Hermitian polynomial matrices, which are defined as those satisfying $[P(\la)]^* = P(\overline \la)$. Equivalently, they are Hermitian for $\la\in \BR$ or have Hermitian coefficients $P_i^*=P_i$. They have a set of eigenvalues $\Lambda$ that is symmetric with respect to the real axis~: $\Lambda=\overline \Lambda$.
  \item Skew-Hermitian polynomial matrices, which are defined as those satisfying $[P(\la)]^* = -P(\overline \la)$. Equivalently, they are skew-Hermitian for $\la\in \BR$ or have skew-Hermitian coefficients $P_i^*=-P_i$. They have a set of eigenvalues $\Lambda$ that is symmetric with respect to the real axis~: $\Lambda= \overline \Lambda$.
  \item Para-Hermitian polynomial matrices, which are those satisfying  $[P(\la)]^* = P(-\overline \la)$. Equivalently, they are Hermitian for $\la\in \jmath\BR$, i.e., for $\la$ on the imaginary axis, or have scaled Hermitian coefficients $P_i^*=(-1)^iP_i$, i.e., $P_i$ is Hermitian for $i$ even and skew-Hermitian for $i$ odd. They have a set of eigenvalues $\Lambda$ that is symmetric with respect to the imaginary axis~: $\Lambda=-\overline \Lambda$.
  \item Para-skew-Hermitian polynomial matrices, which are defined as those satisfying $[P(\la)]^* = -P(-\overline \la)$. Equivalently, they are skew-Hermitian for $\la\in \jmath\BR$ or have skew-Hermitian scaled coefficients $P_i^*=(-1)^{(i+1)}P_i$, i.e., $P_i$ is Hermitian for $i$ odd and skew-Hermitian for $i$ even. They have a set of eigenvalues $\Lambda$ that is symmetric with respect to the imaginary axis~: $\Lambda=- \overline \Lambda$.
 \end{enumerate}

When the polynomial matrix $P(\la)$ has real coefficients $P_i$, these conditions become conditions on the transpose of each $P_i$, and the polynomial matrices are said to be symmetric, skew-symmetric, para-symmetric and para-skew-symmetric. We emphasize that the nomenclature above is the one commonly used for structured rational matrices in the linear system and control theory literature (see, for instance, \cite{GHNSVX02,ran2,ran1} and the references therein). However, in the literature focused on polynomial matrices, para-Hermitian and para-symmetric polynomial matrices are called $*$-even and T-even, respectively, while para-skew-Hermitian and para-skew-symmetric polynomial matrices are called $*$-odd and T-odd, respectively \cite{GoodVibrations}, and all of them are called generically alternating polynomial matrices \cite{MMMM-alternating}.

We first point out that the block Hankel matrix $T$ defined in \eqref{toeplitz} and its compression \eqref{eq.compressT}
inherit particular properties from the self-conjugate structures defined above.

\begin{lemma}  \label{lem:pol} Let $P(\la)\in\BC[\la]^{m\times m}$ be a polynomial matrix as in \eqref{poly}. Let us define the scaling matrix $S:=\diag{(-1)^{(d-1)}I_m, \ldots, (-1)^2I_m, -I_m}$. Then the block Hankel matrix $T$ in \eqref{toeplitz} satisfies the following equations
 \begin{enumerate}
 \item for Hermitian  $P(\la)$: $P_i^*=P_i$ and $T^*=T$,
 \item for skew-Hermitian $P(\la)$: $P_i^*=-P_i$ and $T^*=-T$,
 \item for para-Hermitian $P(\la)$: $P_i^*=(-1)^iP_i$ and $(ST)^*=ST$,
 \item for para-skew-Hermitian $P(\la)$: $P_i^*=(-1)^{(i+1)}P_i$ and $(ST)^*=-ST$.
 \end{enumerate}
The left and right transformations $U$ and $V$ of Theorem \ref{deflate} can then be chosen as $U=V$ in the Hermitian and skew-Hermitian cases and as $U=SV$ in the para-Hermitian and para-skew-Hermitian cases.
\end{lemma}
\begin{proof}
The symmetries of the coefficient matrices $P_i$ trivially yield the four types of symmetries of $T$ mentioned in the Lemma.
For the compression \eqref{eq.compressT}, we can then choose $U=V$ in the Hermitian and skew-Hermitian cases because $T$ is normal, and we can choose $U=SV$ in the para-Hermitian and para-skew-Hermitian cases because $ST$ is then normal.
\end{proof}

Lemma \ref{lem:pol} implies that in the decomposition of Theorem \ref{deflate}, it suffices to construct a transformation $V$ that compresses the columns of $T$ in order to obtain a rank $r$ factorization
\begin{equation} \label{rankr}
U^*TV =\left[\begin{array}{cc} 0 &  0 \\ 0 & \widehat T \end{array}\right] =\left[\begin{array}{cc} 0 &  0 \\ 0 & U_2^*TV_2 \end{array}\right]
\end{equation}
where $\widehat T$ is $r\times r$ and is invertible.
This then leads to the following theorem.

\begin{theorem} \label{Pol} Let $P(\la)\in\BC[\la]^{m\times m}$ be a polynomial matrix as in \eqref{poly}, with a Hermitian, skew-Hermitian, para-Hermitian or para-skew-Hermitian structure. Let $U,V$ be the unitary matrices appearing in \eqref{eq.compressT}, where $U=V$ in the Hermitian and skew-Hermitian cases, and $U=SV$ in the para-Hermitian and para-skew-Hermitian cases with $S:=\diag{(-1)^{(d-1)}I_m, \ldots, (-1)^2I_m, -I_m}$. Then the linear polynomial system matrix
 $$ \widehat L_s(\la):=\left[\begin{array}{c|c} \widehat A_s(\la) & -\widehat B_s(\la) \\ \hline \phantom{\Big|} \widehat C_s(\la) & \widehat D_s(\la) \end{array}\right],$$
 defined in Theorem \ref{deflate},
is a strongly minimal linearization of $P(\la)$ with the same self-conjugate structure as $P(\la)$.
\end{theorem}

\begin{proof}
Let us denote the original pencil  in \eqref{lancaster} as
$$  L_s(\la)=L_0 + \la L_1, $$
then we have the following properties in the four self-conjugate cases~:
 \begin{enumerate}
 \item for Hermitian $P(\la)$, $$L_0^*=L_0, \quad \mathrm{and} \quad L_1^*=L_1,$$
 \item for skew-Hermitian $P(\la)$, $$L_0^*=-L_0, \quad \mathrm{and} \quad L_1^*=-L_1,$$
 \item for para-Hermitian $P(\la)$,
          $$L_0^*\diag{S,I_m}=\diag{S,I_m}L_0, \quad \mathrm{and} \quad L_1^*\diag{S,I_m}=-\diag{S,I_m}L_1,$$
 \item for para-skew-Hermitian $P(\la)$,
  $$L_0^*\diag{S,I_m}=-\diag{S,I_m}L_0, \quad \mathrm{and} \quad L_1^*\diag{S,I_m}=\diag{S,I_m}L_1.$$
 \end{enumerate}
If we choose in the first two cases $U=V$, and in the last two cases $U=SV$ then we obtain for the transformed pair of matrices
$$\tilde L_0:=\diag{U^*,I_m}L_0\diag{V,I_m}, \quad \mathrm{and} \quad   \tilde L_1:=\diag{U^*,I_m}L_1\diag{V,I_m}$$
the properties
 \begin{enumerate}
 \item for Hermitian $P(\la)$, $$\tilde L_0^*=\tilde L_0, \quad \mathrm{and} \quad \tilde L_1^*=\tilde L_1,$$
 \item for skew-Hermitian $P(\la)$, $$\tilde L_0^*=-\tilde L_0, \quad \mathrm{and} \quad \tilde L_1^*=-\tilde L_1,$$
 \item for para-Hermitian $P(\la)$,
           $$\tilde L_0^*=\tilde L_0, \quad \mathrm{and} \quad \tilde L_1^*=-\tilde L_1,$$
 \item for para-skew-Hermitian $P(\la)$,
  $$\tilde L_0^*=-\tilde L_0, \quad \mathrm{and} \quad \tilde L_1^*=\tilde L_1,$$
 \end{enumerate}
 and moreover, their first $m(d-1)-r$ columns and rows are zero because of \eqref{rankr}.
 The pencil $\tilde L_s(\la)$ thus has the same self-conjugate structure as $P(\la)$, and so does the deflated pencil $\widehat L_s(\la)$.
 The strong minimality follows from Theorem \ref{deflate}.
\end{proof}

\begin{remark} \label{rem.symmtheorem} If in Theorem \ref{Pol}, the polynomial matrix $P(\la) \in \mathbb{R} [\la]^{m \times m}$ has real coefficients and is symmetric, skew-symmetric, para-symmetric or para-skew-symmetric, then we can take $V$ and $U$ real. So, $\widehat{L}_s (\la)$ has also real coefficients and the same structure as $P(\la)$.
\end{remark}

The use of {\em strongly minimal linearizations} allows us to prove in Theorem \ref{Pol} a much stronger result for the considered classes of structured polynomial matrices than those available in the literature for {\em GLR-strong linearizations} for the same structures (see, for instance, \cite{bueno-structured,DeTDM,PartI,hmmt2006,GoodVibrations,MMMM-alternating}). We emphasize that Theorem \ref{Pol} provides a simple recipe for constructing a structure preserving strongly minimal linearization for {\em any polynomial matrix} with the considered structures. In contrast, such generality is not possible by using GLR-strong linearizations. For instance, it is shown in \cite{GoodVibrations,MMMM-alternating} that there exist para-symmetric and para-skew-symmetric polynomial matrices of {\em even degree} for which there do not exist any GLR-strong linearization with the same structure. Moreover, it is proved in \cite[Section 7]{DeTDM} that structure preserving GLR-strong linearizations of size $dm \times dm$ do not exist for $m \times m$ polynomial matrices $P(\la)$ of {\em even degree} $d$ possessing any of the self-conjugate structures studied in this paper whenever $P(\la)$ has an {\em odd} number of left minimal indices and an {\em odd} number of right minimal indices. Note that $dm \times dm$ is precisely the size of all explicitly and easily constructible families of GLR-strong linearizations available in the literature for $m\times m$ matrix polynomials of degree $d$, as, for instance, ``vector space linearizations'' \cite{fassbender-saltenberger,hmmt2006,mmmm2006,GoodVibrations}, ``Fiedler-like linearizations'' \cite{Greeks2,bueno-structured,bueno-fiedler-like} and ``block Kronecker linearizations'' \cite{DLPV18,PartI}. Thus, the result in \cite[Section 7]{DeTDM} proves that there exist self-conjugate structured polynomial matrices that cannot be linearized in these families in a structure preserving way. In fact, some of these families of GLR-strong linearizations present other drawbacks as, for instance, the ones considered in \cite{hmmt2006,GoodVibrations} are not valid for any singular polynomial matrix \cite{DeTDM2009} or require certain nonsingularity conditions on the coefficients of the polynomial matrix or on the location of its eigenvalues.

The following example illustrates the discussion in the previous paragraph.
\begin{example} Consider the para-symmetric regular polynomial matrix
\begin{equation} \label{eq.parasymmexample}
P (\la) =
\la^2 \begin{bmatrix}
        1 & 0 \\
        0 & 0
      \end{bmatrix}
- \begin{bmatrix}
        0 & 0 \\
        0 & 1
      \end{bmatrix} = \la^2 P_2 + P_0.
\end{equation}
It is proved via a neat rank argument in \cite[Example 1.4]{MMMM-alternating} that $P(\la)$ does not admit any para-symmetric GLR-strong linearization nor any para-skew-symmetric GLR-strong linearization. However, by using Theorem \ref{Pol} we can easily construct a para-symmetric strongly minimal linearization of $P(\la)$ as follows. Note that, in this case, $T= P_2$, $S = -I_2$ and we can take $V = \begin{bmatrix}
        0 & 1 \\
        1 & 0
      \end{bmatrix}
      $
and $U = -V$. Thus,
\begin{align*}
\mathrm{diag} (U^T, I_2)  \,  L_s(\la) \, \mathrm{diag} (V, I_2) & =
\mathrm{diag} (U^T, I_2)  \, \left[\begin{array}{c|c} -P_2 & \la P_2 \\  \hline \phantom{\Big|} \la P_2 &   P_0 \end{array}\right]  \, \mathrm{diag} (V, I_2)\\
& =
\left[\begin{array}{rr|rr}
0 & 0 & 0 & 0\\
0 & 1 & -\la & 0\\ \hline
0 & \la & 0 & 0\\
0 & 0 & 0 & -1\\
\end{array}\right]
\end{align*}
and
$$
\widehat{L}_s (\la) = \left[\begin{array}{r|rr}
 1 & -\la & 0\\ \hline
 \la & 0 & 0\\
 0 & 0 & -1\\
\end{array}\right]
$$
is a para-symmetric strongly minimal linearization of $P(\la)$. It is obvious that $P(\la)$ and $\widehat{L}_s (\la)$ have the same finite eigenvalues with the same partial multiplicities because both have $\la = 0$ as unique finite eigenvalue with partial multiplicity $2$. Observe that $P(\la)$ has also an eigenvalue at infinity with partial multiplicity $2$. This infinite structure can be recovered from $\widehat{L}_s (\la)$ by using Theorem \ref{theo:recover_infinity} or, better, the simpler Corollary \ref{cor.quadraticinfty} specific for quadratic polynomials. For this purpose, note that, with the notation of Corollary \ref{cor.quadraticinfty}, $r_P = 2$ and $r_2 = 1$. Moreover, $\mathrm{rev}_1 \widehat{L}_s (\la)$ has only one partial multiplicity at $0$ equal to $1$. So, $u = 1$ in Corollary \ref{cor.quadraticinfty}, which yields that $\mathrm{rev}_2 P(\la)$ has $2$ as its unique partial multiplicity at $0$. Note that $P(\la)$ in \eqref{eq.parasymmexample} is also symmetric. It is easy to construct infinitely many symmetric GLR-strong linearizations of $P(\la)$ \cite{hmmt2006}. All of them have necessarily size $4 \times 4$. In contrast, the symmetric strongly minimal linearization constructed by using Theorem \ref{Pol} has size $3 \times 3$.
\end{example}

\begin{remark}  We remark that the procedure presented in Theorem \ref{deflate} has an interpretation in terms of strongly minimal linear realizations of strictly proper rational matrices that have all its poles at $0$. To see that, we apply the change of variable $\la = 1/\mu$ to the system matrix
		$$  \left[\begin{array}{c|c}  \la E - F & \la G \\ \hline \phantom{\Big|}  \la H & 0 \end{array}\right]:=  \left[\begin{array}{c|c}  \widehat A_s(\la) & -\widehat B_s(\la) \\ \hline  \phantom{\Big|} \widehat C_s(\la) & 0 \end{array}\right]
		$$ and we multiply it by $\mu$. Then, we obtain a new linear polynomial system matrix
		$$\left[\begin{array}{c|c}   E - \mu F &  G \\ \hline \phantom{\Big|}   H & 0 \end{array}\right],$$
whose transfer function matrix is
		\begin{equation} \label{eq.remminral}
		H (\mu F- E)^{-1}G = P_2 \, \mu^{-1}  +  P_3 \,  \mu^{-2} + \cdots +  P_d \, \mu^{-(d-1)} .
		\end{equation}
		It can be proved that the new system is also strongly minimal. This means that the triple $\{  E - \mu F, -G,H\}$ is a strongly minimal linear realization of the strictly proper transfer function in \eqref{eq.remminral}, which has all its poles at $\mu =0$. Moreover, the minimal degree of $\det (E-\mu F)$ is known to be $r=\rank{T}$,  since, according to \eqref{eq.widehatAs}, $\widehat A_s (0) = -F = -\widehat T$ is nonsingular.
	\end{remark}

In this section and in the previous one, we have focused on polynomial matrices. We give a general procedure for the construction of strongly minimal linearizations of arbitrary strictly proper rational matrices in Section \ref{sec:sproper}.

\section{Constructing strongly minimal linearizations of strictly proper rational matrices} \label{sec:sproper}
Strictly proper rational matrices $R_{sp}(\la)\in \BC(\la)^{m\times n}$ can be represented via a Laurent expansion around the point at infinity~:
\begin{equation}\label{eq:sproper}
 R_{sp}(\la):= R_{-1} \la ^{-1} + R_{-2} \la ^{-2} +  R_{-3} \la ^{-3} + \cdots .
\end{equation}
In this section, we obtain strongly minimal linearizations for strictly proper rational matrices represented as in \eqref{eq:sproper} by using the algorithm in \cite[Section 3.4]{real}, as we explain in the sequel. Let the block Hankel matrix $H$ and shifted block Hankel matrix  $H_\sigma$ associated with $R_{sp}(\la)$ be denoted as
\begin{equation} \label{Hankel} H :=  \left[\begin{array}{cccc} R_{-1} & R_{-2}  & \ldots & R_{-k}  \\ [+2mm] R_{-2}  &  & \adots &  R_{-k-1} \\   [+2mm]
\vdots & \adots & \adots & \vdots \\  [+2mm] R_{-k} &  R_{-k-1}  & \ldots & R_{-2k+1}
\end{array}\right], \;  H_\sigma :=  \left[\begin{array}{cccc} R_{-2} & R_{-3}  & \ldots & R_{-k-1}  \\ [+2mm] R_{-3}  &  & \adots &  R_{-k-2} \\   [+2mm]
\vdots & \adots & \adots & \vdots \\  [+2mm] R_{-k-1} &  R_{-k-2}  & \ldots & R_{-2k}
\end{array}\right].
\end{equation}
Then for sufficiently large $k$ the rank $r_f$ of $H$ equals the total polar degree of the finite poles, i.e., the sum of the degrees of the denominators in the Smith-McMillan form of $R_{sp}(\la)$ \cite{Kai80}. We assume in the sequel that we are taking such a sufficiently large $k$. The general theory for Hankel based realizations of proper rational matrices in \cite[Section 3.4]{real}, implies that the following Rosenbrock linear system matrix in Theorem \ref{th:lin_sproper} is a strongly minimal linearization for the strictly proper rational matrix $R_{sp}(\la)$.

\begin{theorem} \label{th:lin_sproper}
	Let $R_{sp}(\la)\in \BC(\la)^{m\times n}$ be a strictly proper rational matrix as in \eqref{eq:sproper}. Let $H$ and $H_\sigma$  be the block Hankel matrices in \eqref{Hankel} and $r_f:=\rank H$. Let $U:=\left[\begin{array}{cc}U_1 & U_2\end{array}\right]$ and $V:=\left[\begin{array}{cc}V_1 &  V_2\end{array}\right]$ be unitary matrices such that
	\begin{equation} \label{rankrf}
	U^*HV =\left[\begin{array}{cc} \widehat H &  0 \\ 0 & 0 \end{array}\right] =\left[\begin{array}{cc} U_1^*HV_1 &  0 \\ 0 & 0 \end{array}\right],
	\end{equation}
	where $\widehat H$ is $r_f \times r_f$ and invertible. Let us now partition the matrices $U_1$ and $V_1$ as follows
	\begin{equation} \label{partition} U_1=\left[\begin{array}{cc} U_{11} \\ U_{21} \end{array}\right], \quad \mathrm{and} \quad V_1=\left[\begin{array}{cc} V_{11} \\ V_{21} \end{array}\right],
	\end{equation}
	where the matrices $U_{11}$ and $V_{11}$ have dimension $m\times r_f$ and $n\times r_f$, respectively. Then
	\begin{equation} \label{realize} L_{sp}(\la):= \left[\begin{array}{c|c} U_1^*H_\sigma V_1 - \la \widehat H  &  \widehat H V_{11}^* \\ \hline \phantom{\Big|} U_{11}\widehat H & 0 \end{array}\right]
	\end{equation} is a strongly minimal linearization for $R_{sp}(\la)$. In particular, $R_{sp}(\la) =  U_{11}\widehat H(\la \widehat H- U_1^*H_\sigma V_1)^{-1}\widehat H V_{11}^*.$ \end{theorem}

\begin{remark} \label{rem.state-space}
The algorithm in \cite[Section 3.4]{real} actually constructs a minimal state-space realization $R_{sp} (\la) = \widetilde C (\widetilde A - \la I)^{-1} \widetilde B$ in the classical sense of Rosenbrock, which can be easily transformed  into a minimal state-space realization of the form $R_{sp} (\la) =C (A-\la E )^{-1} B$ with $E$ invertible. We emphasize that,  given {\em any} minimal state-space realization, the associated linear polynomial system matrix
$\begin{bmatrix} A- \la E   & -B \\C & 0\end{bmatrix}$ is a strongly minimal linearization of $R_{sp} (\la)$ due to the minimality  of the realization and the invertibility of $E$.
\end{remark}

\section{Constructing strongly minimal linearizations of self-conjugate strict\-ly proper rational matrices} \label{sec:selft-conjugate-sproper}	
 If a strictly proper rational matrix $R_{sp}(\la)\in\BC(\la)^{m \times m}$ has one of the four self-conjugate structures considered in this paper, we use Theorem \ref{th:lin_sproper} and the ideas developed in \cite[Remark 6.6, Remark 8.4]{DMQ} to construct strongly mininimal self-conjugate linearizations for $R_{sp}(\la)$. First notice that the block Hankel matrices $H$ and $H_\sigma$ in \eqref{Hankel} have the following self-conjugate property depending on that of $R_{sp}(\la)$.
\begin{lemma}  \label{lem:rat} Let $R_{sp}(\la)\in\BC(\la)^{m \times m}$ be a strictly proper rational matrix as in \eqref{eq:sproper}. Let us define the scaling matrix $S:=\diag{-I_m, (-1)^2I_m,\ldots,(-1)^{k}I_m}$. Then the block Hankel matrices $H$ and $H_\sigma$ in \eqref{Hankel} satisfy the following equations
 \begin{enumerate}
 \item for Hermitian $R_{sp}(\la)$:  $R_{-i}^*=R_{-i}$, $H^*=H$ and $H_\sigma^*=H_\sigma$,
 \item for skew-Hermitian $R_{sp}(\la)$: $R_{-i}^*=-R_{-i}$, $H^*=-H$ and $H_\sigma^*=-H_\sigma$,
 \item for para-Hermitian $R_{sp}(\la)$:  $R_{-i}^*=(-1)^iR_{-i}$, $(SH)^*=-SH$  and $(SH_\sigma)^*=SH_\sigma$,
 \item for para-skew-Hermitian $R_{sp}(\la)$:  $R_{-i}^*=(-1)^{(i+1)}R_{-i}$, $(SH)^*=SH$  and $(SH_\sigma)^*=-SH_\sigma$.
 \end{enumerate}
For each of these cases, the left and right transformations $U$ and $V$  in \eqref{rankrf} can be chosen as $U=V$ in the Hermitian and skew-Hermitian cases and as $U=SV$ in the para-Hermitian and para-skew-Hermitian cases.
\end{lemma}
\begin{proof}
The symmetries of the coefficient matrices $R_i$ trivially yield the four types of symmetries of $H$ and $H_\sigma$. For the rank compression \eqref{rankrf}, we can then choose $U=V$ in the Hermitian and skew-Hermitian cases because $H$ is normal, and we can choose $U=SV$ in the para-Hermitian and para-skew-Hermitian cases because $SH$ is then normal.
\end{proof}

We now can use Lemma \ref{lem:rat} and Theorem \ref{th:lin_sproper} to obtain the following structured strongly minimal linearizations of strictly proper rational matrices for the four considered structures.
\begin{theorem} \label{Rsp}
Let $R_{sp}(\la)\in \BC(\la)^{m\times m}$ be a strictly proper rational matrix as in
\eqref{eq:sproper}, and let $H$ and $H_{\sigma}$ be the associated block Hankel matrices appearing in \eqref{Hankel}. Let $U,V$ be the unitary matrices appearing in \eqref{rankrf}, where $U=V$ if $R_{sp}(\la)$ is Hermitian or skew-Hermitian, and $U=SV$ if $R_{sp}(\la)$ is para-Hermitian or para-skew-Hermitian, and where $S:=\diag{-I_m, (-1)^2I_m,\ldots,(-1)^{k}I_m}$. Finally, let $U_1, V_1$ be the $km\times r_f$ matrices formed by the first $r_f$ columns of $U$ and $V$, respectively, $\widehat H$ be the $r_f\times r_f$ matrix defined in \eqref{rankrf} and $U_{11},V_{11}$ be the $m\times r_f$ matrices defined in \eqref{partition}.
\begin{enumerate}
 \item If $R_{sp}(\la)$ is Hermitian, then $\widehat H^*=\widehat H$, $H_\sigma^* =  H_\sigma$, and
$$ L_{sp}(\la):=  \left[\begin{array}{c|c} V_1^*H_\sigma V_1 - \la \widehat H  &  \widehat H V_{11}^* \\ \hline \phantom{\Big|} V_{11}\widehat H & 0 \end{array}\right] $$
is a Hermitian strongly minimal linearization of $R_{sp}(\la)$.
 \item If $R_{sp}(\la)$ is skew-Hermitian, then  $\widehat H^*=-\widehat H$, $H_\sigma^* =  -H_\sigma$, and
$$  L_{sp}(\la):=   \left[\begin{array}{c|c} V_1^*H_\sigma V_1- \la \widehat H  &  \widehat H V_{11}^* \\ \hline \phantom{\Big|} V_{11}\widehat H & 0 \end{array}\right] $$
is a skew-Hermitian strongly minimal linearization of $R_{sp}(\la)$.
 \item If $R_{sp}(\la)$ is para-Hermitian, then  $\widehat H^*=-\widehat H$, $(SH_\sigma)^* =  SH_\sigma$, and
$$ L_{sp}(\la):=    \left[\begin{array}{c|c} V_1^*SH_\sigma V_1 - \la \widehat H  &  \widehat H V_{11}^* \\ \hline \phantom{\Big|} -V_{11}\widehat H & 0 \end{array}\right] $$
is a para-Hermitian strongly minimal linearization of $R_{sp}(\la)$.
 \item If $R_{sp}(\la)$ is para-skew-Hermitian, then  $\widehat H^*=\widehat H$, $(SH_\sigma)^* = -SH_\sigma$, and
$$ L_{sp}(\la):=    \left[\begin{array}{c|c} V_1^*SH_\sigma V_1 - \la \widehat H &  \widehat H V_{11}^* \\ \hline \phantom{\Big|} -V_{11}\widehat H & 0 \end{array}\right] $$
is a para-skew-Hermitian strongly minimal linearization of $R_{sp}(\la)$.
\end{enumerate}
\end{theorem}
\begin{proof}
\begin{enumerate}
 \item If $R_{sp}(\la)$ is Hermitian, then Lemma \ref{lem:rat} and \eqref{rankrf} imply that $H^*=H$, $H_\sigma^* =  H_\sigma$, $U_1=V_1$, and $U_{11}=V_{11}$. The result then follows from $\widehat H=V_1^*HV_1=\widehat H^*$ and \eqref{realize}.
 \item If $R_{sp}(\la)$ is skew-Hermitian, then Lemma \ref{lem:rat} and \eqref{rankrf} imply that $H^*=-H$, $H_\sigma^* =  -H_\sigma$, $U_1=V_1$, and $U_{11}=V_{11}$. The result then follows from $\widehat H=V_1^*HV_1=-\widehat H^*$ and \eqref{realize}.
 \item If $R_{sp}(\la)$ is para-Hermitian, then Lemma \ref{lem:rat} and \eqref{rankrf} imply that $(SH)^*=-SH$, $(SH_\sigma)^* =  SH_\sigma$, $U_1=SV_1$, and $U_{11}=-V_{11}$. The result then follows from $\widehat H=V_1^*SHV_1=-\widehat H^*$ and \eqref{realize}.
 \item If $R_{sp}(\la)$ is para-skew-Hermitian, then Lemma \ref{lem:rat} and \eqref{rankrf} imply that  $(SH)^*=SH$, $(SH_\sigma)^* = -SH_\sigma$, $U_1=SV_1$, and $U_{11}=-V_{11}$. The result then follows from $\widehat H=V_1^*SHV_1=\widehat H^*$ and \eqref{realize}.
\end{enumerate}\end{proof}

\section{Constructing strongly minimal linearizations of arbitrary and self-conjugate rational matrices} \label{sec:rational}
	
	For any given rational matrix $R(\la)\in \BC(\la)^{m\times n}$ we assume that we have an additive decomposition into its polynomial part $P(\la)$ and its strictly proper part $R_{sp}(\la)$ as in \eqref{eq.polspdec}. That is,
	\begin{equation} \label{additive}  R(\la) = P(\la)  + R_{sp}(\la).
	\end{equation}
 For the Laurent expansion given in \eqref{Laurent}, this corresponds to
	$$ P(\la) := R_0 + R_1 \la + \cdots + R_d\la^d, \quad R_{sp}(\la):= R_{-1} \la ^{-1} + R_{-2} \la ^{-2} +  R_{-3} \la ^{-3} + \cdots
	$$
In this section, we obtain strongly minimal linearizations for arbitrary and structured rational matrices by combining strongly minimal linearizations for both parts. The construction for the polynomial part was given in Sections \ref{sec:linearization_polynomial} and \ref{sec:self-cong-poly}, and for the strictly proper part in Sections \ref{sec:sproper} and \ref{sec:selft-conjugate-sproper}.

We emphasize that both for the unstructured case and for the self-conjugate structured cases considered in this paper, we provide a completely general construction of strongly minimal linearizations, which in the structured cases preserve the structure. This means that the construction is valid for any rational matrix, unstructured or with the considered structures. To the best of our knowledge, this generality improves all the results previously available in the literature for the studied self-conjugate structures, as, for instance, those in \cite{GHNSVX02,das-alam-affine,DasAlamArXiv,DMQ}, which do not construct structured linearizations for all structured rational matrices. An important restriction in this setting appears in the recent works \cite{das-alam-affine,DasAlamArXiv,DMQ}, where for those structured rational matrices whose polynomial part has even degree, structure preserving strong linearizations are constructed only if the leading coefficient $R_d$ is nonsingular, which implies that $R(\la)$ must be regular. The use of the new concept of strongly minimal linearization is the reason why our approach is fully general.
	
Once we have strongly minimal linearizations for both the polynomial part and the strictly proper part of a given rational matrix, it is straightforward to construct a strongly minimal linearization for the sum, as shown below.

		\begin{theorem}\label{th:arbitrary_rat}
			Let $R(\la)\in \BC(\la)^{m\times n}$ be an arbitrary rational matrix, i.e., regular or singular. Let $R(\la)=P(\la) +R_{sp}(\la)$ where $P(\la)$ is the polynomial part of $R(\la)$ and $R_{sp}(\la)$ is the strictly proper part of $R(\la)$. Let
			\begin{equation}
				\widehat L_s(\la):=  \left[\begin{array}{c|c}  \widehat A_s(\la) & - \widehat B_s(\la) \\ \hline \phantom{\Big|}  \widehat C_s(\la) & \widehat D_s(\la) \end{array}\right], \quad \mathrm{and} \quad L_{sp}(\la):=  \left[\begin{array}{c|c} A_{sp}(\la) & - B_{sp}(\la) \\ \hline \phantom{\Big|}  C_{sp}(\la) & 0 \end{array}\right],
				\end{equation}
				be strongly minimal linearizations of $P(\la)$ and $R_{sp}(\la)$ as described in Theorems \ref{deflate} and \ref{th:lin_sproper}, respectively. Then
				\begin{equation}
				L(\la):=  \left[\begin{array}{cc|c} \widehat A_s(\la) & 0 & - \widehat B_s(\la) \\  0 & A_{sp}(\la) & -B_{sp}(\la) \\ \hline \phantom{\Big|} \widehat C_s(\la) &  C_{sp}(\la) & \widehat D_s(\la) \end{array}\right]
				\end{equation}
			is a strongly minimal linearization of $R(\la)$.
		\end{theorem}
		\begin{proof}
			The transfer function of $L(\la)$ is clearly
		$$  \widehat D_{s}(\la) + \widehat C_{s}(\la)  \widehat A_{s}(\la)^{-1}  \widehat B_{s}(\la)  + C_{sp}(\la) A_{sp}(\la)^{-1}B_{sp}(\la) = P(\la)+
				R_{sp}(\la)   = R(\la).
				$$
		The strong minimality of $L(\la)$ follows from the fact that the subsystems $ \widehat L_s(\la)$ and $L_{sp}(\la)$ are strongly minimal and have no common poles. Observe, in particular, that $\widehat A_s (\la)$ is unimodular and that the first degree coefficient of $A_{sp} (\la)$ is invertible.
\end{proof}
	
	For the construction of self-conjugate strongly minimal linearizations of structured rational matrices we proceed in the same way and obtain the following result.
		\begin{theorem} \label{th:arbitrary_rat_struct}
			Let $R(\la)\in \BC(\la)^{m\times m}$ be an arbitrary rational matrix, i.e., regular or singular, which has one of the following structures~: Hermitian, skew-Hermitian, para-Hermitian or para-skew-Hermitian. Let $R(\la)=P(\la) +R_{sp}(\la)$ where $P(\la)$ is the polynomial part of $R(\la)$ and $R_{sp}(\la)$ is the strictly proper part of $R(\la)$. Let
		\begin{equation}
			\widehat L_s(\la):=  \left[\begin{array}{c|c}  \widehat A_s(\la) & - \widehat B_s(\la) \\ \hline \phantom{\Big|}  \widehat C_s(\la) & \widehat D_s(\la) \end{array}\right], \quad \mathrm{and} \quad L_{sp}(\la):=  \left[\begin{array}{c|c} A_{sp}(\la) & - B_{sp}(\la) \\ \hline \phantom{\Big|}  C_{sp}(\la) & 0 \end{array}\right],
			\end{equation}
			be strongly minimal linearizations of $P(\la)$ and $R_{sp}(\la)$ as described in Theorems \ref{Pol} and \ref{Rsp}, respectively, according to the corresponding structure of $R(\la)$. Then
			\begin{equation}
			L(\la):=  \left[\begin{array}{cc|c} \widehat A_s(\la) & 0 & - \widehat B_s(\la) \\  0 & A_{sp}(\la) & -B_{sp}(\la) \\ \hline \phantom{\Big|} \widehat C_s(\la) &  C_{sp}(\la) & \widehat D_s(\la) \end{array}\right]
			\end{equation}
			is a strongly minimal linearization of $R(\la)$ with the same self-conjugate structure as $R(\la)$.
		\end{theorem}
		\begin{proof} $L(\la)$ is a strongly minimal linearization of $R(\la)$ by using the same proof as that of Theorem \ref{th:arbitrary_rat}. That the self-conjugate structure of $R(\la)$ and $L(\la)$ are the same, follows from the fact that the self-conjugate structures of $ \widehat L_s(\la)$ and $L_{sp}(\la)$ coincide with that of $R(\la)$. \end{proof}

\begin{remark} Note that Remark \ref{rem.state-space} implies that in Theorems \ref{th:arbitrary_rat} and \ref{th:arbitrary_rat_struct}, one can use the linear polynomial system matrix associated with {\em any} minimal state-space realization of $R_{sp} (\la)$ (with the required structure in the structured case) as the strongly minimal linearization $L_{sp} (\la)$. That is, though Theorems \ref{th:arbitrary_rat} and \ref{th:arbitrary_rat_struct} refer to Theorems \ref{th:lin_sproper} and \ref{Rsp} for constructing the strongly minimal linearizations of $R_{sp}(\la)$, which are based on the Laurent expansion of $R_{sp} (\la)$ around infinity, any other method for constructing minimal state-space realizations of $R_{sp} (\la)$ is valid as well.
\end{remark}

\section{Algorithmic aspects}\label{sec:alg}

The constructive proofs given in the earlier sections in fact lead to possible algorithms for computing strongly minimal linearizations of rational matrices, provided the Laurent expansion \eqref{Laurent} is given up to the term $R_{-2k}$, or any minimal state-space realization is available for the strictly proper part (with the adequate structure in the structured cases). The two decompositions that are required for the construction of the linearizations from the Laurent expansion are the rank factorizations of the constant matrices $T$ in \eqref{eq.compressT} and $H$ in \eqref{rankrf}. Moreover, only the right transformation has to be computed in the self-conjugate structured cases considered in this paper.

It is worth pointing out also that both factorizations only require to construct unitary transformations that ``compress'' the rows and columns of a given matrix, which can be obtained by two $QR$ factorizations in the unstructured case and by only one in the structured cases. Highly efficient algorithms that exploit the special block-Hankel structure (which under block-row reversion becomes block-Toeplitz) of $T$ and $H$ can be found in the literature \cite{HeinigRost,vandooren-laurent-1979}.
We have not imposed conditions on the normal rank of the rational matrix nor on its size (it may be square or rectangular) nor on the ranks of the coefficients of the Laurent expansion, although the self-conjugate structures impose that the rational matrices are square in the considered structured cases. We emphasize that this lack of extra conditions is in contrast with any other previous approach in the literature for linearizing structured rational and polynomial matrices.

Once a strongly minimal linearization of a rational matrix is constructed, the results in Subsection \ref{subsec.stronglyminlinear} guarantee that it contains the complete list of structural data of the rational matrix and that the left/right eigenvectors or the left/right minimal bases of the rational matrix can be very easily recovered from those of the linearization. Thus, the complete information of the rational matrix can be computed by applying to the linearization standard classical algorithms for regular \cite{moler-stewart} or singular \cite{Van79b} generalized eigenvalue problems. In addition, in the self-conjugate structured cases, one can use some of the structured algorithms that have been developed for structured generalized eigenvalue problems as, for instance, those in \cite{skew-staircase,Implicit_palQR,Jacobi_Hermitian,Krylov_symmetric,corner3,Schroder_thesis}. The use of structured algorithms is usually more efficient and has the key advantage of preserving in floating point arithmetic the symmetries of the zeros and poles, the relationships between left and right eigenvectors, the fact that the left and right minimal indices are equal to each other and that the left and right minimal bases are closely related to each other.

\section{Conclusions and future work} \label{sec:conclusions}
In this paper we looked at strongly minimal  linearizations for {\em any} given rational matrix $R(\la)$, preserving the structure whenever we have a specific type of self-conjugate structure on $R(\la)$. We showed that there always exist strongly minimal  linearizations that have the same self-conjugate structure as $R(\la)$. These results were known for the case of proper rational matrices \cite{DMQ,GHNSVX02}, but were extended here to rational matrices that are not proper. Moreover, the derivation is new and is based on arguments that are very similar for the strictly proper part and the polynomial part of the rational matrix. The proofs are also constructive and lead to efficient algorithms for the construction of strongly minimal linearizations both in the unstructured and in the structured cases. The fact that the proposed construction of structure preserving {\em strongly minimal} linearizations is valid for {\em any} structured rational matrix improves significantly all the results available in the literature for constructing other classes of structured preserving linearizations of the structured rational and polynomial matrices considered in this paper, which are not valid for all rational or polynomial matrices (see, for instance,  \cite{das-alam-affine,DasAlamArXiv,DeTDM,DMQ,PartI,GoodVibrations} and the references therein). Interesting future work motivated by the results in this paper may include the extension of the construction of structured strongly minimal linearizations to other classes of structured rational or polynomial matrices and to rational matrices whose polynomial part is expressed in bases different from the monomial basis.


\bibliographystyle{siamplain}

\end{document}